\documentclass{amsart}
\usepackage{mathrsfs}
\usepackage[all]{xy}

\usepackage{mathtools}
\usepackage{mathbbol}
\usepackage{wasysym}
\usepackage{amssymb}

\usepackage{MnSymbol}
\usepackage{comment}
\usepackage{enumerate}
\usepackage{xcolor}

\usepackage{ushort}
\usepackage{enumitem}

\usepackage{changepage}
\usepackage{algorithm}
\usepackage[noend]{algpseudocode}
\makeatletter
\def\BState{\State\hskip-\ALG@thistlm}
\makeatother

\usepackage{ifpdf}
\ifpdf 
  \usepackage[pdftex]{graphicx}
  \DeclareGraphicsExtensions{.pdf,.png,.jpg,.jpeg,.mps}
  \usepackage{pgf}
\else 
  \usepackage{graphicx}
  \DeclareGraphicsExtensions{.eps,.bmp}
  \usepackage{pgf}
  \usepackage{pstricks}
\fi
\usepackage{epic,bez123}
\usepackage{wrapfig}

\usepackage{soul}
\usepackage{url}
\usepackage{tikz}

\newtheorem{thm}{Theorem}[section]
\newtheorem{prop}[thm]{Proposition}
\newtheorem{cor}[thm]{Corollary}
\newtheorem{lem}[thm]{Lemma}

\newtheorem{claim}[thm]{Claim}

\theoremstyle{definition}
\newtheorem{defn}[thm]{Definition}

\theoremstyle{remark}
\newtheorem{quest}[thm]{Question}
\newtheorem{rem}[thm]{Remark}
\newtheorem{example}[thm]{Example}

\newcommand{\U}{ X}
\newcommand{\bU}{\overline{ X}}
\newcommand{\hU}{\partial_h{ X}}

\newcommand{\act}{\curvearrowright}

\newcommand{\p}{\mathcal P}
\newcommand{\f}{\mathscr F}

\newcommand{\PC}{\mathcal P_K(\mathscr F)}

\newcommand{\ax}{\mathrm{Ax}}
\newcommand{\diam }[1]{{\textbf{diam}\big(#1\big)}}

\usepackage[bookmarks=true, pdfauthor={YANG Wenyuan}]{hyperref}

\usepackage[margin=1.1in]{geometry}
\usepackage[normalem]{ulem}

\newtheoremstyle{query}%
{}{}
{\color{red}}
{}
{\sffamily\bfseries}{:}{12pt}
{}
\theoremstyle{query}

\begin{document}

\title[An extreme  boundary of acylindrically hyperbolic groups]{An extreme  boundary of acylindrically hyperbolic groups}

\author{Wenyuan Yang}

\address{Beijing International Center for Mathematical Research\\
Peking University\\
 Beijing 100871, China
P.R.}
\email{wyang@math.pku.edu.cn}

\thanks{(W.Y.) Partially supported by National Key R \& D Program of China (SQ2020YFA070059) and  National Natural Science Foundation of China (No. 12131009 and No.12326601)}
 

\date{\today}

\dedicatory{}


\begin{abstract} 
We prove that every acylindrically hyperbolic group admits a minimal and  extremely proximal  action on a compact metrizable space. If there are no nontrivial finite normal subgroups, then the action is topologically free. This answers positively a question of Ozawa and the applications to $C^\ast$-algebras are discussed. 

\end{abstract}

\maketitle

\section{Introduction}

Let $G$ be a discrete group. A \textit{compact $G$-space $Z$} is a compact topological space $Z$ equipped with a continuous action of $G$ by homeomorphism. It is called \textit{minimal}   if any $G$-orbit is dense in $Z$, and  \textit{strongly proximal} if  any $G$-orbit of a probability measure on $Z$ contains a Dirac measure in its weak closure. A compact $G$-space that is both minimal and strongly proximal is called \textit{$G$-boundary} in the sense of  Furstenberg. Equivalently, $Z$ is a $G$-boundary if, when identified with the set of Dirac measures, it is the unique minimal closed $G$-invariant subset of the compact space of probability measures on $Z$.   The \textit{Furstenberg  boundary} is a universal $G$-boundary so that every $G$-boundary is a continuous $G$-equivariant quotient of it.  We refer to \cite{Oza14} for more details.
\begin{defn}
A compact $G$-space $Z$ is called \textit{extremely proximal} in \cite{Gla74}  if for any compact set $F\neq Z$ and non-empty open set $O$ there is an element $g\in G$ such that $gF\subset O$. A minimal and extremely proximal $G$-space is called \textit{extreme boundary}  (or \textit{strong boundary} in \cite{LS96}) for $G$. 
\end{defn} 

Topological dynamics  on the $G$-boundary have important implications to $C^\ast$-algebra properties. 
A discrete group $G$ is said to be \textit{$C^\ast$-simple} if the reduced group $C^\ast$-algebra $C^\ast_r(G)$ is simple, i.e., has no non-trivial two-sided closed ideals.  A compact $G$-space $Z$  is \textit{topologically free} in \cite{LS96} (or \textit{slender} in \cite{HJ11}) if  the set of points with trivial stabilizer called \textit{free} points is  dense in $Z$. In \cite{KK17}, Kalantar and Kennedy  characterized  the $C^\ast$-simplicity by the topological freeness of the $G$-action on Furstenberg boundary. By the universality, it suffices to find one topological free $G$-boundary to validate the $C^\ast$-simplicity. 

Recently, Robert \cite{Ro25} introduced the notion of a $C^\ast$-selfless property, which among others implies $C^\ast$-simplicity, stable rank one, and strict comparison. This property has attracted considerable recent interest among researchers, e.g. the breakthrough work of Amrutam, Gao, Kunnawalkam Elayavalli, and Patchell in \cite{AGKEP}. Specifically, under the assumption of the rapid decay property, they proved that every acylindrically hyperbolic group without nontrivial finite kernel is $C^\ast$-selfless. Whether rapid decay property could be dropped was an open question. Our work is motivated by the following recent theorem of Ozawa \cite{Oza25}.

\begin{thm}\cite[Theorem 1]{Oza25}\label{Thm: Ozawa}
An infinite countable discrete group $G$ having a topologically-free extreme
boundary is $C^\ast$-selfless.    
\end{thm}

The class of acylindrically hyperbolic groups was introduced by Osin \cite{Osin6} (see \cite{DGO}) and have been intensively investigated in the last decade. See Definition \ref{defn: acylindrical action}. We refer to the survey \cite{OsinICM} for an overview of numerous examples and recent developments. In view of Theorem \ref{Thm: Ozawa}, Ozawa raised the question in  \cite{Oza25} whether every acylindrically hyperbolic group admits an extreme boundary. Many geometrically occurring groups   admit extreme boundaries; just to name a few, Gromov boundary of  hyperbolic groups,  Bowditch boundary of relatively hyperbolic groups, the limit set of convergence groups (\cite{Sun19}),  the visual boundary of CAT(0) groups with rank-one elements (\cite{Ballmann, H09a}), Thurston boundary of mapping class groups (\cite{thurston}), and the Outer space boundary  of $Out(F_n)$ (\cite{LL03}).  Crucially, each of these examples contains elements with north‑south dynamics (see Lemma \ref{loxo North-South dynamics}). A general answer had remained unknown; the main result of this note answers Ozawa’s question affirmatively.

\begin{thm}\label{MainThm}
Let $G$ be a non-elementary acylindrically hyperbolic group.  Then $G$ admits a minimal,  and extremely proximal boundary action on a compact metrizable space. If  $G$ contains no  nontrivial finite normal subgroup, then the action is topologically free.     
\end{thm}


It is a standard fact that if   some  element in $G$ admits north–south dynamics on a minimal $G$-space $Z$, then $Z$ must be an extremely proximal boundary. In what follows, we will explain the challenges in finding such elements and outline our approach.

\subsection*{The construction of extreme boundary.} By definition, every acylindrically hyperbolic group $G$ acts  non-elementarily on a (typically non-locally compact) hyperbolic space $X$. It is well-known that  every loxodromic element in $G$ exhibits north–south dynamics on the Gromov boundary $\partial X$. However, $\partial X$ is generally not compact since $X$ is not locally compact. On the other hand, Gromov introduced  the \textit{horofunction compactification} that applies to any metric space and does yield a compact space $\overline{X}_h:=X\cup \partial_h X$ (in which $X$ is not necessarily open though). An explicit description of the resulting space is notoriously difficult to obtain in concrete examples, and the $G$-action on $\overline{X}_h$ can be pathological: it is generally neither minimal nor does it satisfy the north--south property for loxodromic isometries. The proof of Theorem \ref{MainThm} lies in overcoming these by producing a compact extreme boundary inside $\overline{X}_h$ for a well-chosen candidate $X$. Let us explain the construction as follows:
\begin{enumerate}
\item We replace $X$ with Bestvina-Bromberg-Fujiwara's projection complex $\mathcal X$ constructed from a $G$-invariant family of loxodromic axis in $X$. See \textsection \ref{subsec: projection complex}.
\item By the work of Maher-Tiozzo \cite{MT}, there is a \textit{local minimum} map $\Pi$ (as described in Lemma \ref{localminimalmap}) from  the horofunction compactification $\overline{ \mathcal X}_h$ to the union of  $ \mathcal X$ and its Gromov boundary $\partial   \mathcal X$. \begin{enumerate}[label=(\roman*)]
    \item 
    The $\Pi$-preimage of $\partial   \mathcal X$ consists of  the horofunctions $\phi$ with minimum at $-\infty$: $\inf_{x\in X} \phi(x)=-\infty$, which we denote by $\partial_h^\infty \mathcal X$.  
    \item 
    The map $\Pi$ is   continuous on $\partial_h^\infty X$,  but sends  horofunctions in  $\overline{ \mathcal X}_h\setminus \partial_h^\infty \mathcal X$ to uniformly bounded (usually non-singleton) subsets in ${ \mathcal X}$. See Example \ref{example-horofunctions}.
\end{enumerate}

\item It is useful to partition the horofunction boundary $\partial_h^\infty X$ into  the finite difference classes, denoted as $[\xi]$ with $\xi\in \partial_h^\infty X$, in which any two horofunctions have finite $L^\infty$-difference. In \cite{YANG22}, a general theory of topological dynamics on the horofunction boundary could be obtained modulo such a finite difference partition. For instance, loxodromic elements (strongly contracting elements there) $g\in G$ admit  north-south dynamics relative to the two fixed classes denoted by $[g^-], [g^+]$. The real challenge is that whether these fixed classes could be made singleton for \textit{certain} elements. See \cite[Example 5.17]{MWY25B} for cubical actions of free groups where certain loxodromic elements have non-singleton fixed classes. 

Now, using favorable properties on quasi-trees  from \cite{BBF} (recalled in \textsection \ref{subsec: projection complex}), we  construct specific loxodromic elements $g \in G$ so that the fixed classes are singletons: $[g^-]=\{g^-\},[g^+]=\{g^+\}$. Moreover, the north–south dynamics on the Gromov boundary could be lifted to the horofunction compactification $\overline{\mathcal{X}}_h$ via  the local minimum map. The desired extreme boundary in Theorem \ref{MainThm} is then obtained as the closure of the $G$-orbit of the fixed points $g^\pm$  of $g$ in $\overline{\mathcal{X}}_h$.  See \textsection \ref{subsec: horofunction boundary of projection complex}.
\end{enumerate}

In passing, we consider the following illustrative example of the horofunction compactification.
\begin{example}\label{example-horofunctions}
Let $X$ be a hyperbolic graph  as depicted in Figure \ref{fig:non-singleton infimum}. There is a map from $X$ to the wedge graph of infinitely many rays, so that each 4-circle   $(o,x,y,z_n)$ in $X$ is sent to the origin identified. Each ray in $X$ determines a horofunction  $b_{\xi_n}$. The sequence of $b_{\xi_n}$ tends to a horofunction $b_\eta$ which admits finite infimum:  $b_\eta (x)= b_\eta(y) = -1 =\inf b_\eta$ and $b_\eta(z)\ge 0$ for any $z\in X\setminus \{x,y\}$.  Thus, $b_\eta$ lies in $\overline{X}_h\setminus (X\cup \partial_h^\infty X)$. In fact, $\overline{X}_h= \{b_\eta\} \cup X\cup \partial_h^\infty X$, where $X$ comprises precisely the isolated points.  Further, $\Pi: \overline{X}_h \to X\cup \partial X$ sends $b_\eta$ to $\{x,y\}$ and acts as an identification map elsewhere.
\begin{figure}
    \centering

\tikzset{every picture/.style={line width=0.75pt}} 

\begin{tikzpicture}[x=0.75pt,y=0.75pt,yscale=-1,xscale=1]

\draw    (222.67,189) -- (175.25,222.5) ;
\draw    (222.67,202.33) -- (222.67,189) ;
\draw    (175.25,222.5) -- (242,195.67) ;
\draw    (222,222.5) -- (242,195.67) ;
\draw  [dash pattern={on 4.5pt off 4.5pt}]  (222,222.5) -- (222,207) ;
\draw    (175.25,222.5) -- (199.33,261) ;
\draw [shift={(199.33,261)}, rotate = 57.97] [color={rgb, 255:red, 0; green, 0; blue, 0 }  ][fill={rgb, 255:red, 0; green, 0; blue, 0 }  ][line width=0.75]      (0, 0) circle [x radius= 3.35, y radius= 3.35]   ;
\draw [shift={(175.25,222.5)}, rotate = 57.97] [color={rgb, 255:red, 0; green, 0; blue, 0 }  ][fill={rgb, 255:red, 0; green, 0; blue, 0 }  ][line width=0.75]      (0, 0) circle [x radius= 3.35, y radius= 3.35]   ;
\draw    (199.33,261) -- (222,222.5) ;
\draw [shift={(222,222.5)}, rotate = 300.49] [color={rgb, 255:red, 0; green, 0; blue, 0 }  ][fill={rgb, 255:red, 0; green, 0; blue, 0 }  ][line width=0.75]      (0, 0) circle [x radius= 3.35, y radius= 3.35]   ;
\draw    (175.25,222.5) -- (202.67,182.33) ;
\draw    (202.67,182.33) -- (210.67,200.33) ;
\draw  [dash pattern={on 4.5pt off 4.5pt}]  (210.67,200.33) -- (222,222.5) ;
\draw    (167.33,183.67) -- (175.25,222.5) ;
\draw    (167.33,183.67) -- (192.67,201) ;
\draw  [dash pattern={on 4.5pt off 4.5pt}]  (192.67,201) -- (222,222.5) ;
\draw    (242,195.67) -- (246.5,116) ;
\draw [shift={(246.67,113)}, rotate = 93.23] [fill={rgb, 255:red, 0; green, 0; blue, 0 }  ][line width=0.08]  [draw opacity=0] (8.93,-4.29) -- (0,0) -- (8.93,4.29) -- cycle    ;
\draw [shift={(242,195.67)}, rotate = 273.23] [color={rgb, 255:red, 0; green, 0; blue, 0 }  ][fill={rgb, 255:red, 0; green, 0; blue, 0 }  ][line width=0.75]      (0, 0) circle [x radius= 3.35, y radius= 3.35]   ;
\draw    (222.67,189) -- (225.23,114.66) ;
\draw [shift={(225.33,111.67)}, rotate = 91.97] [fill={rgb, 255:red, 0; green, 0; blue, 0 }  ][line width=0.08]  [draw opacity=0] (8.93,-4.29) -- (0,0) -- (8.93,4.29) -- cycle    ;
\draw [shift={(222.67,189)}, rotate = 271.97] [color={rgb, 255:red, 0; green, 0; blue, 0 }  ][fill={rgb, 255:red, 0; green, 0; blue, 0 }  ][line width=0.75]      (0, 0) circle [x radius= 3.35, y radius= 3.35]   ;
\draw    (202.67,182.33) -- (205.22,114.66) ;
\draw [shift={(205.33,111.67)}, rotate = 92.16] [fill={rgb, 255:red, 0; green, 0; blue, 0 }  ][line width=0.08]  [draw opacity=0] (8.93,-4.29) -- (0,0) -- (8.93,4.29) -- cycle    ;
\draw [shift={(202.67,182.33)}, rotate = 272.16] [color={rgb, 255:red, 0; green, 0; blue, 0 }  ][fill={rgb, 255:red, 0; green, 0; blue, 0 }  ][line width=0.75]      (0, 0) circle [x radius= 3.35, y radius= 3.35]   ;
\draw    (167.33,183.67) -- (167.97,114) ;
\draw [shift={(168,111)}, rotate = 90.53] [fill={rgb, 255:red, 0; green, 0; blue, 0 }  ][line width=0.08]  [draw opacity=0] (8.93,-4.29) -- (0,0) -- (8.93,4.29) -- cycle    ;
\draw [shift={(167.33,183.67)}, rotate = 270.53] [color={rgb, 255:red, 0; green, 0; blue, 0 }  ][fill={rgb, 255:red, 0; green, 0; blue, 0 }  ][line width=0.75]      (0, 0) circle [x radius= 3.35, y radius= 3.35]   ;
\draw  [dash pattern={on 0.84pt off 2.51pt}]  (174.67,161.67) -- (191.33,161.67) ;

\draw (194.67,263.4) node [anchor=north west][inner sep=0.75pt]    {$o$};
\draw (226,217.07) node [anchor=north west][inner sep=0.75pt]    {$x$};
\draw (162,216.73) node [anchor=north west][inner sep=0.75pt]    {$y$};
\draw (246,186.73) node [anchor=north west][inner sep=0.75pt]    {$z_{1}$};
\draw (121.33,171.4) node [anchor=north west][inner sep=0.75pt]    {$\eta \leftarrow z_{n}$};
\draw (186,169.73) node [anchor=north west][inner sep=0.75pt]    {$z_{3}$};
\draw (224,175.07) node [anchor=north west][inner sep=0.75pt]    {$z_{2}$};
\draw (242,96.4) node [anchor=north west][inner sep=0.75pt]    {$\xi _{1}$};
\draw (218.67,95.73) node [anchor=north west][inner sep=0.75pt]    {$\xi _{2}$};
\draw (198,96.4) node [anchor=north west][inner sep=0.75pt]    {$\xi _{3}$};
\draw (128.67,95.73) node [anchor=north west][inner sep=0.75pt]    {$\eta \leftarrow \xi _{n}$};

\end{tikzpicture}
    \caption{Example of horofunctions with infimum realized at non-singleton sets}
    \label{fig:non-singleton infimum}
\end{figure}    
\end{example}

\subsection*{Extreme boundary from curve graphs and coned-off Cayley graphs.} Many of the features we utilize in the projection complex are, in fact, already present in curve graphs of closed surfaces and coned-off Cayley graphs of relatively hyperbolic groups. In Section \ref{sec: extreme bdry for curves and coned-off}, we carry out the above construction for their horofunction compactifications and document the following finer structural properties.

Let $G$ be the orientation-preserving mapping class group of a closed oriented surface $\Sigma$ with genus greater than $2$. The curve graph $\mathcal C(\Sigma)$ of $\Sigma$ is a Gromov hyperbolic space on which $G$ acts  non-elementarily and acylindrically \cite{MinM2,Bow08}. 

\begin{thm}\label{thm: intro curve graph boundary}(=\ref{thm: curve graph boundary})
There exists  a unique $G$-minimal closed subset (denoted as $\partial G$) in the horofunction compactification $\overline{\mathcal C(\Sigma)}_h$    so that $\partial G\cap \mathcal C(\Sigma)$ is exactly the set of all non-separating curves. Moreover, the local minimum map   restricts to be injective over the Myrberg limit set in $\partial \mathcal C(\Sigma)$.       
\end{thm}
The notion of the Myrberg limit set,  due to P. Myrberg \cite{Myr31} in Fuchsian groups, forms a distinguished class of limit points and is receiving many recent interests.  In \cite{MWY25B}, certain Myrberg limit points are proved to be  free points  in the action of discrete groups on the visual boundary  of CAT(0) spaces and of Coxeter groups on the horofunction boundary. The Hausdorff dimension of Myrberg limit set has been explicitly computed for proper actions on hyperbolic spaces  (\cite{MY25}).

In this paper, we define Myrberg limit points  in the Gromov boundary for possibly non-proper actions (Definition \ref{defn: Myrberg}). The injectivity property stated in the ``moreover" part also holds for the projection complex (Lemma \ref{lem: local minimum map injective Myrberg}). This then implies that the free points appearing in topologically free actions in Theorem \ref{MainThm} are in fact preimages of Myrberg limit points in the Gromov boundary. 

Let $(G,\mathcal P)$ be a non-elementary relatively hyperbolic group, where $\mathcal P$ is a finite collection of infinite subgroups.  Let $X$ denote the Farb's \textit{coned-off Cayley graph} of $G$ with the set $\mathcal C(\mathcal P)$  of newly-added \textit{cone points} (\cite{Farb}). Denote by $\partial_B G$  the Bowditch boundary, which as a set is the union of $\mathcal C(\mathcal P)$ (a.k.a. parabolic points) and $\partial X$ (a.k.a. conical points). We refer to \cite{Bow1} and  \textsection\ref{subsec: coned-off cayley graphs} for details.
\begin{thm}\label{thm: intro coned graph boundary}(\ref{thm: coned graph boundary})
There is a unique $G$-invariant closed subset (denoted as $\partial G$) in the horofunction compactification $\bU_h$ so that $\partial G\setminus \mathcal C(\mathcal P)$ is contained in $\partial_h^\infty X$. Moreover, the identity map on $\mathcal C(\mathcal P)$ extends continuously to the local minimum map from $\partial G$  to  $\partial_B G$.  
\end{thm}
 
The distinctions between the statements of Theorems \ref{thm: intro curve graph boundary} and \ref{thm: intro coned graph boundary} may not be immediately apparent. We wish to emphasize that in Theorem \ref{thm: intro curve graph boundary}, it remains unknown whether $\overline{\mathcal{C}(\Sigma)}_h \setminus \partial_h^\infty \mathcal{C}(\Sigma)$ coincides exactly with $\mathcal{C}(\Sigma)$. More precisely,
\begin{quest}\label{quest: curve graph}
On the curve graph, does there exist a horofunction $\phi:\mathcal C(\Sigma)\to \mathbb R$ so that $\phi\notin \partial_h^\infty \mathcal{C}(\Sigma)$:  the minimum
$\{x\in \mathcal C(\Sigma): \inf \phi =\phi(x)\}$ is non-singleton?
\end{quest}
If this is fortunately false, the local minimum map would be continuous, and a better description of $\partial G$ will be available as in Theorem \ref{thm: intro coned graph boundary}. In particular, this would imply that the Gromov boundary of curve graphs could be compactified by the set of non-separated essential curves under the quotient topology. This approach would yield models that are more efficient than the compact spaces in \cite{DHS17,Hamen25}, as those already include the Gromov boundary of curve graphs.

In light of Theorem \ref{MainThm}, it would be interesting to establish a compact boundary characterization for acylindrically hyperbolic groups.
Notably, geometrically finite convergence groups acting on compact spaces have been used to characterize relatively hyperbolic groups \cite{Bow98, Yaman}, a formulation that has proven useful in many applications.
For non‑compact spaces, a boundary characterization of acylindrically hyperbolic groups has already been provided in \cite{Sun19}.  

\subsection*{$C^\ast$-algebra applications.}
As fore-mentioned, we are motivated by Ozawa's Theorem \ref{Thm: Ozawa} that topologically free  extreme boundary implies the $C^\ast$-selfless. We then obtain  from Theorem \ref{MainThm}.
\begin{cor}\label{MainCorollary}
Let $G$ be a non-elementary acylindrically hyperbolic group. Then $G$ is $C^\ast$-selfless if and only if  $G$ contains no  nontrivial finite normal subgroup.     
\end{cor}
This removes the rapid decay assumption of $G$ in \cite{AGKEP}.
Upon completion of this note, we became aware of a recent update  to the paper \cite{Oza25}, in which Ozawa already independently proves Corollary \ref{MainCorollary}.  In contrast to our method via  his Theorem \ref{Thm: Ozawa}, his approach is based on de la Harpe's axiomatization of Powers' argument \cite{Pow75,HG86} and does not use the compactness. We refer to \cite[Theorem 14]{Oza25} for details. 

In the course of this work, we took a curious look at  several closely-related properties in $C^\ast$-algebras, such as the free semi-group property \cite{DH99}, the $P_{nai}$ property \cite{BCH94}, and paradoxical towers \cite{GGKN}. We use  boundary actions  to provide a concise and elementary proof that these properties hold for a \textit{tamed} class of non-elementary action of $G$ on a hyperbolic space $X$. The notion of a tamed action   (see Definition \ref{defn: tamed action}) includes the class of acylindrically hyperbolic groups. Thus, we  generalize the corresponding results in acylindrically hyperbolic groups were proved in  \cite[Proposition 4.1(a)]{GO20}, \cite[Theorem 0.1]{AD19}, and \cite[Proposition 4.7]{GGKN}. We refer to Propositions \ref{prop: free semgigroups} and \ref{prop: paradoxical towers} for detailed statements.
 

\setcounter{tocdepth}{2} \tableofcontents
\subsection*{Acknowledgments}
We are grateful to Xin Ma for bringing our attention to the reference \cite{Oza25}, and to Kang Li for the reference \cite{GGKN}, as well as for their insightful discussions on $C^\ast$-algebras. Thanks also goes to Adrien Le Boudec for pointing out a short proof of topological freeness  (Remark \ref{rmkmonolith}). 

\section{Preliminaries}

In the first two subsections, we  review the standard  facts in hyperbolic geometry about Gromov boundary and limit sets, and then define the notion of Myrberg limit points. From \textsection \ref{sec: horofunction boundary}, we  introduce  the horofunction boundary, and the finite difference partition on it, and focus on its relation with Gromov boundary. Several elementary results    proved in \textsection \ref{sec: minimal and isolated points} shall be used later on.     

\subsection{Hyperbolic spaces and Gromov boundary}
A geodesic metric space $(X,d)$ is called \textit{proper} if every closed ball is compact. A map $\phi:I\subset \mathbb R\to X$ from an interval $I$ is called a \textit{$c$-quasi-geodesic} for some $c\ge 1$ if $c^{-1}|s-t|-c\le d(\phi(s),\phi(t))\le c|s-t|+c$ for any $s, t\in I$.

We say that $X$  is  \textit{$\delta$-hyperbolic} for some $\delta\ge 0$ if any geodesic triangle is $\delta$-thin. Define the Gromov product as $\langle x,y\rangle_z=(d(x,z)+d(y,z)-d(x,y))/2$ for $x,y,z\in X$. A sequence of points $x_n\in X$ \textit{converges to infinity} if $\langle x_n, x_m\rangle_o\to \infty$ as $n,m\to \infty$. Two such sequences $x_n,y_n$ are  \textit{equivalent} if $\langle x_n, y_m\rangle_o\to \infty$ as $n,m\to \infty$. Let $\partial X$ denote the Gromov boundary  of $X$, which consists of equivalent classes of $x_n\in X$ converging to infinity. We could equip $\partial X$ with a metrizable topology by the so-called visual metric.  A sequence $x_n\in X$ converging to infinity may be represented by a quasi-geodesic ray $\gamma$ whose quasi-geodesicity parameters depend only on $\delta$. 

When $X$ is proper,   $\partial X$ is a compact  space. In general, it  may not be compact, which is  mostly interesting in the present paper.

\subsection{Limit set and Myrberg limit points}\label{sec: limit set and Myrberg} 
Let $G$ be a  group acting (not necessarily properly) by isometry on $X$.
We say that a point $\xi\in \partial X$ is called \textit{limit point} if there exists $g_n\in G$ so that $g_no \to \xi$ for some (or any) $o\in X$ as $n\to\infty$. The set of limit points forms the  \textit{limit set} denoted as $\Lambda G$.

An isometry $g$ is called \textit{loxodromic} if the \textit{quasi-axis}  $\ax(g)=\cup_{n\in \mathbb Z} g^n[o,go]$ is a quasi-geodesic. The  quasi-geodesics $\ax(g)$ defines the  two fixed points in $\partial X$ denoted by $g^-,g^+$, which are the limit points of $g^no$ as $n\to -\infty$ (resp. $n\to +\infty$). It is well-known that loxodromic isometries admit \textit{north-south dynamics} relative to their two fixed points on $\partial X$, stated as follows. 

\begin{lem}\label{loxo North-South dynamics}
Let $g$ be a loxodromic isometry on $X$ with fixed points $g^-\ne g^+$. Then for any open neighborhoods $U$ and $V$ of $g^-$ and $g^+$ respectively, there exists $n_0$ so that $g^n(\partial X\setminus U)\subset V$ and $g^{-n}(\partial X\setminus V)\subset U$ for any $n\ge n_0$.  
\end{lem}
By this reason, $g^-$ and $g^+$  are refereed to as the \textit{repeller} and  \textit{attractor}, which shall be switched for the inverse  $g^{-1}$. We write $g^\pm=\{g^-,g^+\}$.

Two loxodromic isometries are called \textit{weakly independent} if their fixed points are disjoint in $\partial X$. Equivalently, their axis have bounded projection.
An isometric action of $G$ on $X$ is called \textit{non-elementary} if $G$ contains two weakly independent loxodromic isometries. 

\begin{defn}\cite{Osin6}\label{defn: acylindrical action}
We say that a non-elementary action of $G$ on $X$  is \textit{acylindrical} if   for any $r>0$ there exist $L,N<\infty$ with the following property. If $d(x,y)>L$ for $x,y\in X$, then $\sharp \{g\in G: d(x,gx)\le r, d(y,gy)\le r\}\le N$.    
\end{defn}

If the action is acylindrical, then every loxodromic element $g\in G$ is contained in a unique maximal elementary subgroup denoted by $E(g)$ (\cite[Lemma 6.5]{DGO}). Equivalently, $E(g)=\{f\in G: fg^\pm=g^\pm\}$ is exactly the set-stabilizer of the two fixed points $g^-,g^+$ in $\partial X$. Two loxodromic elements $h,k$ are called \textit{independent} if $gE(h)g\ne E(k)$ for any $g\in G$. This is a stronger notion of independence, as explained in the next lemma.

\begin{lem}\label{lem: acylin independent}\cite[Theorem 6.8]{DGO}
Assume that $G$ acts non-elementarily and acylindrically on a hyperbolic space $X$.
If two loxodromic elements $h,k\in G$ are independent, then the family $\f=\{g\ax(h), g\ax(k): g\in G\}$ has bounded projection property. That is, there exists  $\tau>0$ so that for any $\alpha\ne \beta\in \f$, the shortest projection of $\alpha$ to $\beta$ has diameter at most $\tau$.    
\end{lem}



In what follows, assume only that  $G$ acts non-elementarily on a hyperbolic space $X$.

The following elementary lemma will be frequently used. It is refereed to as Extension Lemma in \cite{YANG10}, which holds for weakly independent contracting isometries on a general metric space.
\begin{lem}\label{extend3}
Let $F$ be a set of three weakly independent loxodromic isometries on $X$. Then there exist $c>1, \tau>0$ depending $o$ and the axis $\ax(f)$ with $f\in$ $F$ so that 
$$
\forall f\ne f'\in F:\quad \min\{\pi_{\ax(f)}([o,go]),\pi_{\ax(f')}([o,go])\}\le \tau
$$
for any $g\in G$. 
In particular, for any $g,h\in G$, there exists $f\in F$ so that $gfh$ labels a $c$-quasi-geodesic. 
\end{lem}
\begin{proof}
The displayed equation follows from the fact that the axis of three elements in $F$ has bounded projection. The ``in particular" statement then follows from the well-known fact in hyperbolic geometry that local quasi-geodesics are quasi-geodesics.      
\end{proof}

We remark that a proof of the following well-known fact in a much greater generality could find in \cite[Lemma 3.13]{YANG22} for contracting elements in groups acting on the so-called convergence boundary.
\begin{lem}\label{lem: double dense elements}
Let $g,h$ be two weakly independent loxodromic isometries on $X$. Then for all $n\gg 0$,  $g^nh^n$ is a loxodromic element. Moreover,  the attractor and repeller    of $g^nh^n$ tend to $ g^+$ and $h^-$ respectively.      
\end{lem}

The following fact is also well-known \cite[8.2.H]{Gro}; we include a proof to illustrate several arguments used later on.   
\begin{lem}\label{limitsetbigtwo}
Assume that $G$ acts non-elementarily on $X$. Then 
\begin{enumerate}[label=(\roman*)]
    \item The action of $G$ on the limit set $\Lambda G$ is minimal: any $G$-orbit in $\Lambda G$ is dense. 
    \item 
    The set of fixed point pairs of all loxodromic isometries is dense in $\Lambda G\times \Lambda G$.
\end{enumerate}
\end{lem}
\begin{proof}
Fix a loxodromic isometry $h\in G$ on $X$ with fixed points $h^-,h^+\in \partial X$. Since $h$ admits north-south dynamics on  $\Lambda G$, we see that the closure $\Lambda$ of the set $\{gh^-,gh^+: g\in G\}$ is the unique $G$-minimal subset of $\partial X$. Indeed, let $A$ be a $G$-invariant closed subset in $\partial X$. Then $|A|\ge 3$; otherwise it contradicts the existence of two weakly independent loxodromic isometries. To see $\Lambda\subset A$, we choose a point $x\in A\setminus \{h^-,h^+\}$. As $h^nx\to h^+$, we have $h^+\in A$, so $\Lambda \subset A$. We now prove that $\Lambda =\Lambda G$. It suffices to prove $\Lambda G\subset \Lambda$. For any $p\in \Lambda G$, there exists $g_n\in G$ so that $g_no\to p$. Choose $f\in F$ by Lemma \ref{extend3} so that $[o,g_no]$ has bounded $\tau$-projection to $g_n\ax(f)$ for some $\tau>0$ independent of $g_n$. Then it is easy exercise that $g_nf^+\to p$.  

To prove \textit{(ii)}, we first note that any open subset $U$ contains the fixed points of some loxodromic element $g$. Indeed, let $h,k$ be two weakly independent loxodromic elements. By \textit{(i)},  there is some $f\in G$ so that $(fhf^{-1})^+=fh^+\in U$. By the north-south dynamics for $fhf^{-1}$, since $fh^\pm\cap fk^\pm=\emptyset$,  for all large  $n\gg 1$,  the fixed points $fh^nf^{-1} (fk^\pm)$ of $fh^n k h^{-n}f^{-1}=(fh^nf^{-1}) (fkf^{-1}) (fh^{-n}f^{-1})$ are contained in $U$. This also implies that for any loxodromic element $h$ there are infinitely many loxodromic elements that are weakly independent with $h$.

Given $a\ne b\in \Lambda G$, let us take disjoint open neighborhoods $U$ and $V$ of $a$ and $b$. By the above discussion, there exist two loxodromic isometries $h, k$ so that $h^\pm\in U, k^\pm\in V$. By Lemma \ref{lem: double dense elements}, for $n\gg 1$, $g_n:=h^nk^n$ is loxodromic isometry with fixed points $g_n^+\to h^+$ and $g_n^-\to h^-$. Thus, $(g_n^+,g_n^-)\in (U,V)$. As $U,V$ are arbitrary the double density follows.   
\end{proof}

\begin{defn}\label{defn: Myrberg}
A limit point $p\in \Lambda G$ is called \textit{Myrberg limit point} if $\Lambda G\times \Lambda G$ is contained in the accumulation points of  $\{(gx,gp): g\in G\}$ for some  (or any) $x\in X$. That is, for any $\xi, \eta\in \Lambda G$, there exists $g_n\in G$ so that $g_nx\to \xi$ and $g_np\to \eta$.    
\end{defn}
\begin{lem}\label{lem: Myrberg chracterization}\cite[Lemma 2.11]{MY25}
A  point $p\in \Lambda G$ is a {Myrberg} limit point if and only if for any $c>1$ there exists $R=R(c,\delta)$ with the following property. Given a point $x\in X$, let $[x,p]$ be a $c$-quasi-geodesic ray ending at $\xi$ and   $f\in G$  be  a loxodromic isometry. Then there exists a sequence of $g_n\in G$ so that the diameter of $[x,p]\cap N_R(g_n\ax(f))$  tends to $\infty$.    
\end{lem}

The next fact for the proper action on CAT(0) spaces with rank-one elements  is proved in (\cite[Lemma 3.20]{MWY25B}). Moreover, for a proper action on hyperbolic spaces, the Hausdorff dimension of Myrberg limit set in Gromov boundary is computed as growth rate of the action which is positive (\cite[Theorem 1.10]{MY25}).  
\begin{lem}\label{lem: uncountable Myrberg}
Assume that $G$ acts non-elementarily on a hyperbolic space $X$. Then Myrberg limit points are uncountable in $\Lambda G$.    
\end{lem}
\begin{proof}
The  proof  is identical to \cite[Lemma 3.20]{MWY25B} for CAT(0) spaces, even though here the action  is not assumed to be proper. For completeness, we sketch the construction of uncountable Myrberg points.  Let $\mathcal L=\{h_1,h_2,\cdots, h_n,\cdots \}$ be the set of all loxodromic isometries. We define a map $\mathbb N^\infty\to \Lambda_{m} G$ which is proved to injective. 

Let $\omega=(n_1,n_2,\cdots, n_i,\cdots)\in \mathbb N^\infty$ be a sequence of positive integers. Consider the infinite word  $W=\prod_{i=1}^\infty h_i\cdot f_i^{2n_i}$ where $f_i\in F$ are chosen by Lemma \ref{extend3}. Denoting $g_n=\prod_{i=1}^{n-1} h_i\cdot f_i^{2n_i}$ for $n\ge 2$ and $g_1=1$, we    concatenat  paths labeled by $h_i f_i^{n_i}$ as follows:
$$\gamma=\bigcup_{i\ge 1} \left(g_i [o, h_io][h_io, h_if_i^{2n_i}]\right)$$
which is a $c$-quasi-geodesic ray for some $c\ge 1$.  By Lemma \ref{lem: Myrberg chracterization}, we see  that $\gamma$ ends at a Myrberg limit point.  Hence, it remains to prove  that the assignment 
$$
\begin{aligned}
\Phi: \;  \mathbb N^\infty&\longrightarrow \Lambda G\\
\omega &\longmapsto \Phi(\omega)
\end{aligned}
$$
is injective. We refer to \cite[Lemma 3.20]{MWY25B} for a proof.  
\end{proof}

\subsection{Horofunction boundary and Gromov boundary}\label{sec: horofunction boundary}
We first recall the definition of horofunction boundary and setup the notations for further reference. 

Let $X$ be a (possibly non-proper) metric space with a basepoint $o\in X$. Let $\mathrm{Lip}_1^o(X)$ denote the set of 1-Lipchitz functions that vanish at $o$. The function space $\mathrm{Lip}_1^o(X)$ is compact endowed  with the compact-open topology. For  each $y \in  X$, we define a Lipschitz map $b_y^o:  X \to \mathbb R$ in $\mathrm{Lip}_1^o(X)$    by $$\forall x\in X:\quad b^o_y(x)=d(x, y)-d(o,y).$$ The closure $\overline{X}_h$  of $\{b^o_y: y\in  \U\}$  defines a compact space which we call \textit{horofunction compactification}, in which  $ X$ is dense, but not necessarily open though.  Explicitly,  $\overline{X}_h$ consists of all pointwise limits  of $b_{y_n}$ for some sequence of points $y_n\in X$, which we  denote by $b_\xi$ and write $y_n\to\xi$.  According to the context, we use $\xi$ and $b_\xi$ to denote  the  elements  in $\overline{X}_h$.  The complement, denoted by $\hU$, of $ \U$ in $\bU_h$ is called  the \textit{horofunction boundary}. All elements in $\overline{X}_h$ are called \textit{horofunctions}.  If $X$ is a separable metric space, then $\overline{X}_h$ is metrizable by \cite[Proposition 3.1]{MT}. 

\begin{rem}\label{rem: topology on horofunctions}
The compact-open topology is the same as the pointwise convergence topology on $\mathrm{Lip}_1^o(X)$. When $X$ is proper, it is also the same as the uniform convergence topology on bounded sets. In general, $\mathrm{Lip}_1^o(X)$ with the latter topology may not be  compact.  Let us illustrate this for the example that $X$ is a locally infinite tree.
\begin{enumerate}
    \item 
    The above horofunction compactification $\overline{X}_h$ is the union of $X$ and the ends of $X$. The vertices in $X$ with finite valence are isolated points, while the others are accumulation points of the geodesic rays (i.e. ends of $X$) issuing from it.
    \item 
    The accumulation points of $X$ in $\mathrm{Lip}_1^o(X)$ under the uniform convergence topology on bounded sets are just the ends of $X$. In general, if $X$ is a CAT(0) space, then the accumulation points form exactly the visual boundary of $X$.
\end{enumerate}

\end{rem}


 

The topology of horofunction boundary is independent of  the choice of   basepoints, since if $d(x, y_n)-d(o, y_n)$ converges as $n\to \infty$, then so does $d(x, y_n)-d(o', y_n)$ for any $o'\in \U$. Moreover, the corresponding horofunctions differ by an additive amount: $$b^o_\xi(\cdot)-b_\xi^{o'}(\cdot)=b_\xi^o(o'),$$ so we will omit the upper index $o$. Every isometry $\phi$ of $X$ induces a homeomorphism on $\overline{X}_h$:  
$$
\forall y\in X:\quad\phi(\xi)(y):=b_\xi(\phi^{-1}(y))-b_\xi(\phi^{-1}(o)).
$$

\subsection{Finite difference relation and boundary comparison}\label{sec: finite difference relation}
Two horofunctions $b_\xi, b_\eta$ have   \textit{$K$-finite difference} for some $K\ge 0$ if the $L^\infty$-norm of their difference is $K$-bounded: $$\|b_\xi-b_\eta\|_\infty = \sup_{x\in \U}|b_\xi(x)-b_\eta(x)|\le K.$$ 
We say that $b_\xi,b_\eta$ has finite difference if they have \textit{$K$-finite difference} for some $K$ depending on $\xi,\eta$. 
The   \textit{$[\cdot]$-class} of     $b_\xi$ consists of  horofunctions $b_\eta$ so that $b_\xi, b_\eta$ have   finite difference.  The loci   $[b_\xi]$  of    horofunctions $b_\xi$ form a \textit{finite difference equivalence relation} $[\cdot]$ on $\hU$. The \textit{$[\cdot]$-closure} denoted as $[\Lambda]$ of a subset $\Lambda\subseteq \hU$ is the union of $[\cdot]$-classes of all points in $\Lambda$. We say that $\Lambda$ is \textit{saturated} if $[\Lambda]=\Lambda$.

Note that all horofunctions $b_x$ with $x\in X$ are in the same $[\cdot]$-class denoted by $[X]$. 
The following fact follows directly by definition.
\begin{lem}\label{bddLocusLem}
Let $b_{x_n} \to b_\xi\in \hU$ and  $b_{y_n}\to b_\eta\in \hU$ with $x_n,y_n\in X$. If  $\sup_{n\ge 1}d(x_n, y_n)<\infty$, then  $[b_\xi]=[b_\eta]$. 
\end{lem}

\textbf{Convention.} Until the end of this subsection, assume that $X$ is a $\delta$-hyperbolic graph with each edge assigned unit length. For technical reasons, we consider the horofunction compactification of the vertex set of $X$ so that horofunctions are integer valued. This loses  no essentials since its $[\cdot]$-closure recovers the whole compactification of $X$.

Let $\partial_h^\infty X$ denote the set of horofunctions $b_\xi$ with $\inf(b_\xi)=-\infty$. It is clear that the minimum of $b_x$ with $x\in X$ is realized  at the unique point $x$. Thus, $X$ is disjoint with $\partial_h^\infty X$, so  $\partial_h^\infty X$ is necessarily contained in $\partial_h X=\overline{X}_h\setminus X$, which justifies the notations.  We now recall the local minimum map defined by Maher-Tiozzo \cite{MT}.
\begin{defn}\cite[Definition 3.11]{MT}\label{defn: local minimal map}
The \textit{local minimum} map $\pi: \overline{X}_h\to X\cup \partial X$ is defined as follows.
\begin{enumerate}
    \item If $\xi \in  \overline{X}_h\setminus \partial_h^\infty X$,  then define $$\Pi(\xi):=\{x\in X: b_\xi(x)\le \min b_\xi\}$$
where the minimum exists since  $X$ is assumed to be a graph and $b_\xi$ is integer valued. 
    \item 
    If $\xi \in \partial_h^\infty X$,  then choose a sequence of points $x_n\in X$ so that $b_\xi(x_n)\to -\infty$, and define $$\Pi(\xi):=\lim_{n\to\infty} x_n$$
\end{enumerate}    
\end{defn}
\begin{rem}\label{rem: finite horofunctions}
If $X$ is a proper space, then  $\overline{X}_h= X\cup \partial_h^\infty X$. This equality also holds in many non-locally finite examples  (e.g.  $X$ is a locally infinite tree as noted in Remark \ref{rem: topology on horofunctions}, and is a locally infinite quasi-median graph in \cite[Theorem 4.38]{MWY25B}). In general, it is possible that $\overline{X}_h\setminus (X\cup \partial_h^\infty X)$ is  non-empty.  That is, there exists a horofunction $h$ (e.g. Example \ref{example-horofunctions}) with at least two points realizing the infimum of $h$.  
\end{rem}

\begin{lem}\label{localminimalmap}\cite[Lemma 3.12]{MT}
Assume that $X$ is a possibly non-proper $\delta$-hyperbolic geodesic space for some $\delta>0$. Then there exists a constant $K=K(\delta)$ so that  the local minimum map $\Pi: \overline{X}_h\to X\cup \partial X$ is a  $G$-equivariant map  with the following properties.
\begin{enumerate}
    \item $\Pi: \partial_h^\infty X\to \partial X$ is a $G$-equivariant continuous map, where any two horofunctions in $\Pi^{-1}(p)$ are $K$-bounded for every $p\in \partial X$.
    \item The image $\Pi(b_\xi)$ has  diameter at most $K$ for any $\xi\in \overline{X}_h\setminus \partial_h^\infty X$.
\end{enumerate}
\end{lem} 

In the remainder of this subsection, we fix  a loxodromic isometry  $g$ on $X$. The preimages under $\Pi$ of the two fixed points of $g$ in $\partial X$ gives the two  $[\cdot]$-classes $[g^-]\ne [g^+]$ in $\partial_h^\infty X$ which are preserved by $g$. Let $\ax(g)$ denote the quasi-axis and $\pi_{\ax(g)}: X\to \ax(g)$ the shortest projection map. We wish to extend the map to the boundary points in $\hU$. 

\begin{lem}\label{lem: shortest projection for bdry points}
There exists a constant $D=D(\delta)>0$ with the following property. 

Consider    $\xi,\eta$ in $\partial_h^\infty X$ with $[\xi]=[\eta]$, and let $x_n,y_n\in X$   so that $x_n\to \xi$ and $y_n\to \eta$. Then for some large $n_0$, the set  $\{\pi_{\ax(g)}(x_n),\pi_{\ax(g)}(y_n): n\ge n_0\}$ has  diameter at most   $D$.      
\end{lem}
\begin{proof}
By Lemma \ref{localminimalmap},  $|B_\xi(u,v)-B_\eta(u,v)|\le K$ for any $u,v\in X$. By the contracting property of $\ax(g)$,   there is a constant $R$ depending on the quasi-geodesicity constant of $\ax(g)$ with the following property. If $d_{\ax(g)}(x,y)>R$,  then $[x,y]$ intersects the $R$-neighborhood of $\ax(g)$ with diameter greater than $d_{\ax(g)}(x,y)-R$. Let $L$ denote the diameter of a fundamental domain for $\langle g\rangle$ on $\ax(g)$. Setting $D:=K+5R+2L$, we shall prove $d_{\ax(g)}(x_n,y_n)\le D$ for any $n\gg 1$.  

Indeed, assume by   contradiction that $d_{\ax(g)}(x_n,y_n)> K+5R+2L$ for infinitely many $n\ge 1$. By the above discussion, $[x_n,y_n]\cap N_R(\ax(g))$   has diameter greater than $K+4R+2L$. Up to  $\langle g\rangle$-translation of $[x_n,y_n]$, we may assume that there are two points $u,v\in \ax(g)$ so that $d(u,v)>4K-R$ and $u,v\in N_R([x_n,y_n])$ for these $n$. Let $u_n,v_n\in [x_n,y_n]$ so that $d(u,u_n),d(v,v_n)\le R$ and $d(x_n,u_n)<d(x_n,v_n)$. By definition,
$$B_\xi(u,v)=\lim_{x_n\to \xi} d(u,x_n)-d(v,x_n)\le -d(u,v)+2R$$ and $$B_\eta(u,v)=\lim_{x_n\to \xi} d(u,x_n)-d(v,x_n)\ge d(u,v)-2R$$
which yields $|B_\xi(u,v) - B_\eta(u,v)| \ge 2d(u,v)-4R >K$. This is a contradiction. The lemma is proved.
\end{proof}

\begin{defn}\label{defn: bdry projection for loxo}
We define a projection map $\pi_{\ax(g)}: \bU_h\setminus [g^\pm]\to \ax(g)$ as follows.

If $\xi\in \bU_h\setminus \partial_h^\infty X$,  define $\pi_{\ax(g)}(\xi)=\pi_{\ax(g)}(x)$ for some $x\in \Pi(\xi)$.

If $\xi\in \partial_h^\infty X$, fix a sequence of points $z_n\in X\to [\xi]$ and  define $\pi_{\ax(g)}(\xi):=\pi_{\ax(g)}(z_n)$ for any large $n\gg 1$.
\end{defn}
Of course, the map $\pi_{\ax(g)}$  depends on the choice of $x\in \Pi(\xi)$, and the sequence $\{z_n: n\ge 1\}$, so it is coarsely defined: for different choices, the projection differs up to finite uniform neighborhood. Even though, a certain coarse continuity holds as follows.
\begin{cor}\label{lem: projection continuity}
Let $D$ be as in Lemma \ref{lem: shortest projection for bdry points}.
Let $\xi\in \bU_h\setminus [g^\pm]$ and $x_n\in X \to \xi$. Then for any large enough $n\gg 1$,  $\pi_{\ax(g)}(x_n)$ is contained in the $2D$-neighborhood of $\pi_{\ax(g)}(\xi)$.     
\end{cor}
\begin{proof}
The case   for boundary points  $\xi\in \partial_h^\infty X\setminus  [g^\pm]$ holds by Lemma 
\ref{lem: shortest projection for bdry points}. For the case $\xi\in X$ there is nothing to do, as $\pi_{\ax(g)}(x_n)$ is exactly the shortest projection of points in $X$. It remains to consider the case $\xi\in \partial_hX\setminus \partial^\infty X$. We recall $\Pi(\xi)=\{x\in X: b_\xi(x)\le \min b_\xi\}$ and $\Pi(x_n)=\{x_n\}$. Since $b_{x_n}\to b_{\xi}$ pointwise and horofunctions are integer-valued, it is straightforward to check that $x_n$ is eventually contained in $\Pi(\xi)$ so the conclusion follows as well in this case.
\end{proof}

\subsection{Recognizing the minimal and isolated points}\label{sec: minimal and isolated points}
In this subsection, assume that $X$ is a general metric graph.
We give a few elementary criteria regarding  which points in $\bU_h$ are minimal and isolated.  We say that a boundary point $\xi\in \partial_h X$ is \textit{minimal} if  its equivalent class $[\xi]$ is singleton.
The vertices  with finite degree are clearly isolated. In a locally infinite graph, dead ends are also isolated.

\begin{defn}
Let $o,x\in X$ be two points. We say that $x$ is a \textit{dead end} relative to $o$ if no geodesic from $o$ to $x$ could extend further: there is  no point $y\in X$ so that $d(o,y)=d(o,x)+1$.      
\end{defn}

\begin{lem}\label{lem: dead end is isolated}
Assume that $x\in X$ is a dead end relative to some $o\in X$. Then the horofunction $b_x$ is an isolated point in $\overline X_h$.    
\end{lem}
\begin{proof}
The horofunction on a graph is integer-valued. If for some $x\in X$, $b_x$ is an accumulation point, then $b_{y_n}(x)=b_x(x)$ for a sequence of $y_n$ in $X$. This however implies $d(o,x)+d(x,y_n)=d(o,y_n)$, which contradicts that $x$ is a dead end.    
\end{proof}

We now give a criterion for $[\cdot]$-classes being minimal. 
\begin{defn}\label{defn: guardian}
Let $x,y,z$ be three points in $X$. We say that $y$ is a \textit{guard} between $x,z$ if every geodesic from $x$ to $z$ passes through $y$. 

If $z$ is a  boundary point in $\partial_h X$,  we say that $y$ is a \textit{guard} between $x,z$ provided that  for any $z_n\in X\to z$, $y$ is a guard between $x,z_n$ with all large $n\gg 0$.    
\end{defn}

\begin{lem}\label{lem: minimal class criterion}
Let $(x_n)$ and $(y_n)$ be two sequences of points in $X$ so that $x_n\to \xi\in \hU$ and  $y_n\to \zeta\in \hU$.  Assume that for any $z\in X$ there exists a large integer $n_0$ (depending on $z$) so that each $x_n$ with $n\ge n_0$ is a guard from $z$ to $y_m$ for all but finitely many $y_m$. Then $\zeta=\xi$.    
\end{lem}
\begin{proof}
We fix a basepoint $o$.
By assumption,  $x_n$ lies on two geodesics $[o,y_m]$ and $[z,y_m]$ for $m\gg 0$.
By definition, $b_\xi(z)=\lim_{n\to\infty} d(z,x_n)-d(o,x_n)$ and $b_\zeta(z)=\lim_{n\to\infty} d(z,y_n)-d(o,y_n)$.  Then $d(z,x_n)-d(o,x_n)=d(z,y_m)-d(o,x_m)$.  This shows the conclusion.  
\end{proof}

\begin{lem}\label{lem: accumulation point criterion}
Let  $\xi_n$ be a distinct sequence of boundary points in $\partial_h X$. Assume that $v$ is a vertex  so that given any $y\in X$, $v$ is a guard between $y$ and $\xi_n$ for all $n\gg 0$. Then $v$ is an accumulation point of $\xi_n$ in $\bU_h$.     
\end{lem}
\begin{proof}
The proof is similar to the one of Lemma \ref{lem: minimal class criterion}. We leave it to the interested reader.   
\end{proof}

\section{Extreme boundary from projection complex}\label{SSecProjectionComplex}

In this section, we  first recall  the work of Bestvina-Bromberg-Fujiwara  \cite{BBF} on the construction of  a quasi-tree. We  then study its horofunction boundary and  construct the loxodromic elements with north-south dynamics on it, which leads to a proof of Theorem \ref{MainThm}. 

\subsection{Projection complex}\label{subsec: projection complex}

\begin{defn}[Projection axioms]\cite{BBF}
\label{defn:projaxioms}
Let $\f$ be a collection of   metric spaces equipped with (set-valued) projection maps  $$\Big\{\pi_{U}:  \f\setminus \{U\}\to U\Big\}_{U\in \f}$$  Denote $d_{U}(V,W) := \mathrm{diam}(\pi_{U}(V) \cup \pi_{U}(W))$ for $V\ne U\ne W \in \f$.  The pair $(\f, \{\pi_{U}\}_{U\in \f})$ satisfies {\it projection axioms} for a {   constant} $\kappa \ge 0$ if
 \begin{enumerate}
     \item
     \label{axiom1}  $\mathrm{diam}(\pi_{U}(V)) \le \kappa$ when $U \neq V$.
     \item
     \label{axiom2} If $U,V,W$ are distinct and $d_{V}(U, W) > \kappa$ then $d_{U}(V,W) \le \kappa$.
     \item
     \label{axiom3} The set $\{ U \in \f \,:\, d_{U}(V, W) > \kappa \}$ is finite for $V \neq W$.
 \end{enumerate}
 \end{defn}

Assume that $G$ admits a non-elementary acylindrical action on a hyperbolic space $X$. Let $F$ be a set of three independent loxodromic isometries on $X$. Let $\f=\{g\ax(f): g\in G, f\in F\}$ be the  system of all $G$-translated $f$-axis of $f\in F$. By Lemma \ref{lem: acylin independent}, $\f$ satisfies the projection axioms. 
 
For any $X,U,V,W\in \f$,   the following triangle inequality holds:
\[
d_X(V,W)\le d_X(V,U)+d_X(U,W).
\]

Given   $K>0$ and $V, W\in \f$, define the interval-like set $$\f_K(V,W):=\{U\in\f: d_U(V,W)>K\}.$$ 


\begin{defn}
The \textit{projection complex} $\p_K(\f)$  is  a graph with the vertex set consisting of the elements in $\f$. Two  vertices $U$ and $V$ are connected if $\f_K(U,V)=\emptyset$. We equip $\PC$ with a length metric $d_{\p}$ induced by assigning unit length to each edge.
\end{defn}

From now on, we  fix a   large enough constant  $K$ relative to $\kappa$,  so the projection complex $\p_K(\f)$ is connected by  \cite[Proposition 3.7]{BBF}.  Moreover, it is a quasi-tree of infinite diameter on which $G$ acts non-elementarily by \cite[Theorem 3.16, Corollary 3.16]{BBF}. We even know that $G$ acylindrically on $\p_K(\f)$ by  \cite[Theorem 5.6]{BBFS}, though we do not need it below. 


The following two facts will be very useful to analyze the horofunction boundary.

\begin{lem}\label{ForcingLem}\cite[Lemma 3.18]{BBF}
For any $K\gg 0$ there exists $\hat K>K$ such that $\f_{\hat K}(U,V)$ is contained in the vertex set of any geodesic  from $U$ to $V$ in $\PC$. That is, every element in $\f_{\hat K}(U,V)$ is a guard  between  $U$ and $V$ in sense of Definition \ref{defn: guardian}.
\end{lem}


\begin{lem}\cite[Proposition 3.14]{BBF}\label{lem: BGIT}
There exists $K_0$ with the following property.
Assume that $\{U_0,U_1,\cdots,U_k\}$ is a  path in $\PC$ and $V$ is a vertex with $d(U_i,V)\ge 3$ for all $i$. Then  $d_V(U_0,U_i)\le K_0$ for any $1\le i\le k$.   
\end{lem}


\subsection{Horofunction boundary of projection complex}\label{subsec: horofunction boundary of projection complex}

This subsection is devoted to the proof of Theorem \ref{MainThm}, which relies on the the following key fact. For simplicity, write $\mathcal X=\p_K(\f)$.  Let $\Pi: \mathcal X\cup\partial_h\mathcal X\to \mathcal X\cup\partial\mathcal X$ be the local minimum map given by Lemma \ref{localminimalmap}.
\begin{prop}\label{prop: minimal loxodromic elements}
There exist infinitely many weakly independent loxodromic isometries in $G$ on $\mathcal X$ with minimal fixed points in the horofunction boundary  $\partial_h \mathcal X$.    
\end{prop}
\begin{proof}
Fix a nontrivial element $g\in G$. By Lemma \ref{extend3}, there exists $f\in F$ so that for any large  $n\gg 0$, $gf^n$ labels a $c$-quasi-geodesic on $X$ for some $c$ depending only on $\f$ (not on $n$). Namely, the path $$\gamma:=\bigcup_{i\in \mathbb Z} gf^n([o,go]\cdot g[o,f^no])$$  is a $c$-quasi-geodesic. Denote $U_i=(gf^n)^i\ax(f)$ with $i\in \mathbb Z$.  Let $\pi\gamma$ denote the concatenated path in the projection complex by connecting vertices $U_i$ and $U_{i+1}$ by a geodesic for each $i\in \mathbb Z$. Choosing equivariantly the geodesics between $U_i$ and $U_{i+1}$, we may assume that $\pi\gamma$ is left invariant under $gf^n$.  

Fix any $i<j<k$. By Morse lemma, the first inequality holds  for a constant $D_0$ depending only on $D_0,\delta$ : $$d_{U_j}(U_i,U_k)\ge d(o,f^no)-D_0  \gg \hat K$$ where the last one holds when $n\gg 1$, so by Lemma \ref{ForcingLem} the interval $\f_K(U_0,U_i)$ is contained in any $\mathcal X$-geodesic  from $U_0$ to $U_i$ for any given $0\ne i\in \mathbb Z$. This shows that $\pi\gamma$  is a  bi-infinite geodesic  in $\mathcal X$. Thus, for all large $n\gg 1$, $gf^n$  is a loxodromic element on $\mathcal X$.

If $\xi$ and $p$ denote the corresponding limit points of $\{U_i\}_i$ in $\partial_h\mathcal X$ and $\partial \mathcal X$ as $i\to+\infty$, then $\Pi(\xi)=p$ by Lemma \ref{localminimalmap}. Note that $p$ is the attractor of  $gf^n$. We next prove that the $\Pi$-preimage of $p$ is a singleton.   The case for the repeller of $gf^n$  is similar.

Indeed, given $\eta\in \Pi^{-1}(p)$, let $\{V_j\}_j$ be a sequence of vertices in $\mathcal X$ so that $V_j\to \eta$. That is, $[\xi]=[\eta]$. Then $\{V_j\}_j$ and $\{U_i\}_i$ converge to the same $p$ in $ \partial \mathcal X$.  Thus, $\langle V_j, U_i\rangle_{U_0}$ tends to $\infty$ as $i, j\to \infty$. Let us fix any  $U_{i_0}$ with $i_0\ge 1$. This implies that  the distance $d(U_{i_0}, [U_i,V_j])$ tends to $\infty$ as    $i, j\to\infty$. By Lemma \ref{lem: BGIT}, $d_{U_{i_0}}(U_i,V_j)\le K_0$ for all large $i, j\gg i_0$.

Fix any vertex $Z$ in the projection complex $\mathcal X$. When $i_0\gg 0$, $d(U_0,U_{i_0})\gg d(U_0,Z)$,  so it follows that $d_{U_{i_0}}(Z,U_0)\le K_0$ by Lemma \ref{lem: BGIT}. Thus, for $j\gg i_0$,  $$d_{U_{i_0}}(Z,V_j)\ge d_{U_{i_0}}(U_0,U_i)-d_{U_{i_0}}(Z,U_0)-d_{U_{i_0}}(U_i,V_j)\ge d(o,f^no)-D_0-2K_0>K$$ By Lemma \ref{ForcingLem}, $U_{i_0}$ is a guard between $Z$ and $V_j$ for all large $j\gg i_0$.  As $Z$ and $U_{i_0}$ are arbitrary, the condition of Lemma \ref{lem: minimal class criterion} is verified, so $\xi=\eta$ follows. Therefore, we proved that the fixed  $[\cdot]$-classes of $gf^n$ for $n\gg 1$ are minimal in $\partial_h \mathcal X$.  Varying $n$ constructs infinitely many weakly independent loxodromic isometries. The proof is then complete. 
\end{proof}

A similar proof shows that the local minimum map $\Pi$ is injective on the preimages of Myrberg limit points in $\partial \mathcal X$. 

\begin{lem}\label{lem: local minimum map injective Myrberg}
Let $p\in \partial \mathcal X$ be a  Myrberg limit point for the action $G\act \partial \mathcal X$.
Then   $\Pi^{-1}(p)$   is a singleton in the horofunction boundary $\partial_h \mathcal X$. 
\end{lem}
\begin{proof}
Fix a base vertex $o\in \mathcal X$.  In the proof of Proposition \ref{prop: minimal loxodromic elements},  for any large $\tilde K$, we constructed   a loxodromic isometry $h\in G$ (eg. $h=gf^n$ for $n\gg 0$) so that $d_U(o,ho)>\tilde K$ for some vertex $U$ on $[o, ho]$. 

Now, let $\gamma$ be a quasi-geodesic ray in $\mathcal X$  starting at $o$ and ending at  $p\in \Lambda G$. By Lemma \ref{lem: Myrberg chracterization},  $\gamma$ contains an infinite sequence of subpaths $p_n$ so that $\mathrm{Len}(p_n)\to \infty$ and $p_n$ lies in the $R$-neighborhood of $g_n\ax(h)$ with $g_n\in G$. Arguing similarly  in the proof of Proposition \ref{prop: minimal loxodromic elements},  one may verify by Lemma \ref{lem: BGIT} that $g_n U$ is a guard between $o$ and any vertex $z$ on $\gamma$ with $d(o,z)\gg d(o, g_nU)$. So by Lemma \ref{lem: minimal class criterion}, we see that $\gamma$ accumulates to the same point $\xi$ in $\partial_h \mathcal X$, and $[\xi]$ is a singleton. 
\end{proof}

Let $\mathcal L$ denote  all the loxodromic isometries $g\in G$ which have minimal fixed points $g^-,g^+$ in $\overline{\mathcal X}_h$. It is non-empty by Proposition \ref{prop: minimal loxodromic elements}. We now define the desired extreme boundary in Theorem \ref{MainThm} to be  the topological closure in $\overline{\mathcal X}_h$ :  
$$\partial G:=\overline{\{g^-,g^+: g\in \mathcal L\}}$$   of the fixed points of elements in $\mathcal L$. Then $\partial G$ is a compact subset, as $\overline{\mathcal X}_h$  is so. 
\begin{lem}\label{lem: minimal loxodromic with ns dynamics}
Any loxodromic isometry $g\in \mathcal L$ admits north-south dynamics relative to $g^\pm$ on $\overline{\mathcal X}_h$. In particular, $\partial G$ is the minimal $G$-invariant closed subset.  
\end{lem}
\begin{proof}
The outline follows the one of \cite[Lemma 3.27]{YANG22} which proves the same conclusion when $X$ is a proper space. In the non-proper case, we need take of the horofunctions in $ \partial_h \mathcal X \setminus \partial_h^\infty\mathcal X$. To deal with them we employ the local minimum map in Lemma \ref{localminimalmap}.

The quasi-axis  $\ax(g)=\bigcup_{i\in \mathbb Z} g^i[o,go]$ is a bi-infinite quasi-geodesic in $\mathcal X$. A \textit{positive half-ray} of $\ax(g)$ refers to a half-ray $\bigcup_{i\ge i_0} g^i[o,go]$ for some $i_0$. The complement of  a positive half-ray is a \textit{negative half-ray}. By assumption of $g\in \mathcal L$, the fixed points $g^+, g^-$ are minimal in  $\partial_h^\infty \mathcal X$.  Thus, any  positive and negative   half-rays  tend to $g^-$ and $g^+$ respectively with respect to the horofunction topology.   

Assuming north-south dynamics, the ``in particular" part could be proved by the same argument as in Lemma \ref{limitsetbigtwo}(i). It is sufficient to prove that $g$ admits north-south dynamics relative to $g^\pm$. To that end, take any open neighborhoods  $\mathcal V$ and $\mathcal U$ of  $g^-, g^+\in \overline{\mathcal X}_h$ respectively. Then $\mathcal W=\overline{\mathcal X}_h\setminus \mathcal V$ is a compact subset.  
\begin{claim}
The projection $\pi_{\ax(g)}(\mathcal W)$ is contained in some positive half-ray of $\ax(g)$. Similarly,  the projection $\pi_{\ax(g)}(\overline{\mathcal X}_h\setminus \mathcal U)$  is contained in some negative half-ray of $\ax(g)$. 
\end{claim}
\begin{proof}[Proof of the claim]
We only prove the first statement; the ``similar" statement is completely analogous. 
Let us assume to the contrary that there exists a sequence of points $w_n\in\mathcal W$ (possibly in $\partial_h\mathcal X$) so that their projections $z_n\in \pi_{\ax(g)}(w_n)$ leave every positive half ray of $\ax(g)$. If $w_n\in \partial_h\mathcal X$ is a boundary point, one may take a sequence of points in $\mathcal X$  tending to $w_n$, so their projections to $\ax(g)$ are uniformly close to $\pi_{\ax(g)}(w_n)$ by Corollary \ref{lem: projection continuity}. Consequently,  we may  assume without loss of generality that $w_n$ are contained in $\mathcal X$.  

By $\delta$-hyperbolicity, the shortest projection point of $w_n$ to $\ax(g)$ is uniformly close to $[o,w_n]$. Thus $[o,w_n]$ intersects the $R$-neighborhood of $z_n$ for some $R$ depending on $\delta$. In particular, $|d(o,w_n) - d(o,z_n)-d(z_n,w_n)|\le 2R$. Given any $x\in \mathcal X$,  the shortest projection of $x$ to $\ax(g)$ is fixed and as $d(o,z_n)\to\infty$, we deduce from hyperbolicity that  $[x,w_n]$ intersects the $R$-neighborhood of $z_n$ for a possibly larger but still denoted by $R$. Thus, $|d(x,w_n) - d(x,z_n)-d(z_n,w_n)|\le 2R$. These   together  show  $|b_{w_n}(x)-b_{z_n}(x)|\le 4R$ for  $x\in \mathcal X$. 

Since $\mathcal W$ is compact, some subsequence of $b_{w_n}$ tends to a horofunction $b_\xi$ for some $\xi \in \mathcal W$.  Note that $z_n\in \ax(g)$ accumulates into $[g^^-]$, so the minimality of $g^-$ implies $b_{z_n}\to b_{g^-}$ and thus $\xi=g^-$. This contradicts the assumption that $g^-\in \mathcal V$.    
\end{proof}

Note that $g$ acts by translation on $\ax(g)$, so  for any given positive half-ray,  its image under a large power $g^n$ is disjoint with a given negative half-ray. Thus, $g^n\mathcal W$ must be contained in $\mathcal U$ by the Claim. This shows that the north-south dynamics of $g$ on $\overline{\mathcal X}_h$.
\end{proof}

We are now ready to complete the proof of Theorem \ref{MainThm} restated below.
\begin{thm}\label{thm:ahextremebdry}
Let $G$ be a non-elementary acylindrically hyperbolic group. Then $G$ admits a minimal and extreme boundary action on a compact metrizable space $\partial G$. Further, if $G$ has trivial finite kernel, then the action  is topologically free.  
\end{thm}
\begin{proof}
The extreme proximal action  follows from the north-south dynamics in Lemma \ref{lem: minimal loxodromic with ns dynamics}. It remains to show the topological free action on $\partial G$. Note that this does not follow   from Abbot-Dahmani's result \cite[Proposition 0.3]{AD19} which only proves topological free action on the Gromov boundary $\partial \mathcal X$. 

By Lemma \ref{lem: local minimum map injective Myrberg}, the set of Myrberg points in $\Lambda G$ are injective into $\partial G$ under the inverse $\Pi^{-1}$ of the local minimum map. The topological freeness on $\partial G$ would follow, if we could prove that there are Myrberg points that are free points on $\Lambda G$. Indeed, this is what was proved for CAT(0) spaces with rank-one elements \cite[Theorem 3.24]{MWY25B}. Its proof generalizes in a straightforward way. We sketch it for completeness. 

If $G$ has trivial finite kernel, there exists  a loxodromic isometry  $f$  in $\partial_h \mathcal X$ so that $E(f)=\langle f\rangle$  by \cite[Theorem 6.14]{DGO}. Moreover, the set $\{g\ax(f):g\in G\}$  has bounded intersection. As $G$ is countable, the fixed points of loxodromic isometries  are countable. By Lemma \ref{lem: uncountable Myrberg}, there exists a Myrberg point $z\in \Lambda G$ that is not fixed by any loxodromic isometries (there are no parabolic isometries in acylindrical actions). We claim that $z$ is a free point in the sense that the stabilizer of $z$ is trivial.  Indeed, if a nontrivial element $g\in G$ fixes $z$, then $g$ must be elliptic. So it fixes a quasi-geodesic ray $\gamma$ from $o$ to $z$ up to a finite Hausdorff distance. By Lemma \ref{lem: Myrberg chracterization}, $\gamma$ intersects a sequence of  $N_R(g_n\ax(f))$ in a diameter tending to $\infty$. This implies that $g$ preserves some axis  $g_n \ax(f)$, so $g$ belongs to $g_nE(h)g_n^{-1}$. Thus, $g$ must be loxodromic. This   contradiction concludes  the proof of topological freeness.      
\end{proof}
\begin{rem}\label{rmkmonolith}
Here is a short proof of the topological freeness, which was pointed out to the author by Adrien Le Boudec.  
Recall that the \textit{monolith} of a group $G$  is the intersection of all non-trivial normal subgroups of $G$. We say $G$ is \textit{monolithic} if the monolith is non-trivial. Any free group is not monolithic, since any subgroups are free. By a theorem of Dahmani-Guirardel-Osin \cite{DGO}, a non-elementary acylindrically hyperbolic group contains many free normal subgroups, so it is not monolithic. Now, let $Z$ be a
faithful, minimal and extremely proximal $G$-boundary. By \cite[Proposition 4.6.b]{LB21}, if the action of $G$ on $Z$ is not topologically free, then $G$ is monolithic. This contradiction shows the topologically free action of Theorem \ref{thm:ahextremebdry}. 

The proof in \cite[Theorem 3.24]{MWY25B} provides insight into free points, which are Myrberg limit points, in a more elementary manner, while Dahmani–Guirardel–Osin’s theorem requires significantly more advanced machinery such as rotating families.    
\end{rem}

\section{Extreme boundaries from curve graphs and coned-off Cayley graphs}\label{sec: extreme bdry for curves and coned-off}

As a case study we examine in this section  the horofunction boundary of curve graphs and of coned-off Cayley graphs. We hope this illustrates some aspects of Theorem \ref{MainThm} in these setups. 

\subsection{Curve graphs of surfaces}
Let $\Sigma$ be an orientable surface of finite type with negative Euler characteristic.
We say a simple closed curve on $\Sigma$ is \textit{essential} if it is not parallel to a boundary component and does not bound a disc or a punctured disk. An essential simple closed curve is \textit{separating} if its complement in $\Sigma$ is disconnected; otherwise it is called \textit{non-separating}. A \textit{separating/non-separating curve} means an isotopy class of separating/non-separating essential simple closed curve. We say that a set of curves  \textit{miss} another set of curves if they admit disjoint representatives, and otherwise we say that they \textit{cut}. Let  $\mathcal C(\Sigma)$ denote the set of all curves with a graph structure by declaring an edge  between two distinct curves that miss. This is the 1-skeleton of the Harvey's  curve complex, which we call \textit{curve  graph}.  Let $G$ be the mapping class group of a closed orientable surface $\Sigma$.  It is well-known that $\mathcal C(\Sigma)$ is a hyperbolic space and  $G$ acts non-elementarily and acylindrically on $\mathcal C(\Sigma)$.

We write \textit{subsurface} to denote a compact, connected, proper subsurface of $\Sigma$ such that each component of its boundary is essential in $\Sigma$. We only define the \textit{subsurface projection} $$\pi_S: \mathcal C(\Sigma)\to \mathcal C(S)$$ for non-annular subsurface $S$, while the one for annular surface is not needed in this paper.   We say that a curve  $\alpha$  \textit{cuts}   $S$, if it is  contained in $S$ or it cuts the boundary $\partial S$ of $S$. If $N$ is a closed tubular neighborhood of the union of $\alpha$ and the components of $\partial S$ that  $\alpha$ cuts,  we define $\pi_S(\alpha)$ to be the  boundary curves contained in $S$ of $N$.  If $\alpha$ misses $S$, set $\pi_S(\alpha)=\emptyset$.   Since $\pi_S(\alpha)$ has diameter at most 2, $\pi_S$ is a 2-Lipschitz map.

If two curves $\beta,\gamma$   cut the subsurface $S$,  we define $$d_{\alpha}(\beta,\gamma)=\mathrm{diam}_{\mathcal C(S)}(\pi_{S}(\beta)\cup \pi_{S}(\gamma))$$
If $S:=\Sigma\setminus \alpha$ for  a non-separating curve $\alpha$, we write $d_{\alpha}(\beta,\gamma):=d_{S}(\beta,\gamma)$.

The following is Masur-Minsky's Bounded Geodesic Image Theorem.

\begin{lem}\label{lem: Masur Minsky BGI}\cite{MinM2}
There exists a constant $D$ depending on genus $g$ with the following property.
Let $S$ be a subsurface in $\Sigma$. If $(\alpha_0,\cdots,\alpha_n)$ is a geodesic in $\mathcal C(\Sigma)$ so that every  curve $\alpha_i$ cuts $S$. Then $d_{S}(\alpha_i,\alpha_j)\le D$ for any $0\le i, j\le n$.    
\end{lem}


Let $(\alpha_0,\cdots,\alpha_n)$ be a path  without backtracking (i.e. $\alpha_{i}\ne \alpha_{j}$ for any $i\ne j$). A non-separating  curve $\alpha_i$  with $1\le i\le n-1$ is called \textit{$K$-local large} for some $K>0$ if $d_{\alpha_i}(\alpha_{i-1},\alpha_{i+1})>K$. 

\begin{lem}\label{lem: guard in curve graph}
For any $c, K\ge 1$, there exists  $L=L(c,K)$ with the following property.
Let $\gamma=(\alpha_{-n},\cdots,\alpha_{-1},\alpha_0,\alpha_1,\cdots, \alpha_m)$ for $n, m\ge 1$ be a $c$-quasi-geodesic with no backtracking. 
If $\alpha_0$ is locally $L$-large in $\gamma$, then $\alpha_0$ is locally $K$-large in any geodesic from $\alpha_{-n}$  to $\alpha_m$.   
\end{lem}
\begin{proof}
Since $\gamma$ is a $c$-quasi-geodesic, we may find $k>0$ depending on $c$ only with the following property.
If $p$ is a geodesic from $\alpha_{-n}$ to $\alpha_{-k}$, and $q$ is a geodesic from $\alpha_{k}$ to $\alpha_m$, which may be empty if $k>n$ or $k>m$, then $d(p,\alpha_0),d(q,\alpha_0)\ge 2$. Thus, each curve on $p$ cuts  $S:=\Sigma\setminus \alpha$, so $\pi_\alpha(\alpha_{-n},\alpha_{-k})\le D$ by Lemma \ref{lem: Masur Minsky BGI}. Similarly,  $\pi_\alpha(\alpha_{2},\alpha_{m})\le D$. Note that $\pi_\alpha$ is 2-Lipschitz, so $\pi_\alpha(\alpha_{-k},\alpha_{-1}),\pi_\alpha(\alpha_{1},\alpha_{k})\le 2k$. Since $d_{\alpha_0}(\alpha_{-1},\alpha_1)\ge L$ by assumption, the triangle inequality for $\pi_\alpha$ shows that $d_\alpha(\alpha_{-n},\alpha_{m})\ge L-2(D+k)=:K$.
Since $\alpha$ is non-separating, any essential curve that is distinct from $\alpha$ has to cut $\Sigma\setminus \alpha$. Choosing $L$ large enough so that $L-2(D+k)>D$,  the pages 910-911 of \cite{MinM2} proves via Lemma \ref{lem: Masur Minsky BGI} that $\alpha$ must appear in any geodesic from $\alpha_{-n}$  to $\alpha_m$.  The conclusion follows.
\end{proof}
\begin{lem}\label{lem: exist K local large sequence}
Let $p$ be a Myrberg limit point for the action of $G$ on $\partial\mathcal C(\Sigma)$.
Then for any $K\ge 0$, there exists a geodesic ray $\gamma$ ending at $\xi$ which contains an  unbounded sequence of $K$-locally large non-separating curves. In particular, the $[\cdot]$-class of a Myrberg limit point is minimal. 
\end{lem}
\begin{proof}
As in Lemma \ref{lem: uncountable Myrberg}, we   construct a geodesic ray $\gamma$ with desired properties ending at a Myrberg point. 
To this end, let $\mathcal L=\{h_1,h_2,\cdots\}$ be all pseudo-Anosov elements in $G$. Fix a base point $o$. For any $L$, fix an element $g\in G$ so that $[o,go]$ contains a local $L$-large vertex $v$. Let us form the word $W=h_1f_1gf_1'h_2f_2gf_2'\cdots h_nf_ngf_n'\cdots$    where $f_i, f_i'$ are chosen in $F$ by Lemma \ref{extend3} for the pair $(h_i,g)$ and $(g,h_{i+1})$. Let $\gamma$ the labeled path by $W$   in the same manner as in  the proof of  Lemma  \ref{lem: uncountable Myrberg}. Then the path $\gamma$ is a $c$-quasi-geodesic ending at a Myrberg point for a constant $c>1$. Taking $L$ large enough, Lemma \ref{lem: guard in curve graph} shows that the corresponding translates of $v$ on $\gamma$  are $K$-locally large. The ``in particular" clause follows by Lemma \ref{lem: accumulation point criterion}.   
\end{proof}

\begin{lem}\label{lem: non-separating curve}
Each non-separating  curve  is an accumulation point of the horofunction boundary $\partial_h \mathcal C(\Sigma)$.    
\end{lem}
\begin{proof}
Let $v$ be a non-separating  curve on $\Sigma$. We are going to find a sequence of boundary points $\xi_n$ so that $\xi_n\to v$. Denote $S=\Sigma\setminus v$. Choose a geodesic ray $p$ starting at $v$ which ends at   $\xi$ in  the horofunction boundary  $\partial_h \mathcal C(\Sigma)$.  Pick a partial pseudo-Anosov element $f$ on $\Sigma$; that is, $f$ restricts to be a pseudo-Anosov element on $S$ and leaves $v$ invariant. We claim that $\xi_n:=f^n\xi\to v$. 

Indeed, let $u\in p$ be the vertex next to $v$, which cuts $S$. As $f\vert_S$ is pseudo-Anosov and acts as a loxodromic isometry on $\mathcal C(S)$, we see that $d_S(u,f^nu)\to\infty$ as $n\to\infty$.  For any curve $v\ne w\in \mathcal C(\Sigma)$, let us consider the concatenated path by the geodesics $[w,v]$ and $f^n[v,\xi]$. Since $d_S(u,f^nu)\to\infty$, a similar argument as in Lemma \ref{lem: guard in curve graph} shows   that for all large enough $n\gg 1$ (depending on $w$), $v$ is a guard from $w$ to $\xi_n$. Then $\xi_n\to v$ follows by Lemma \ref{lem: accumulation point criterion}.  
\end{proof}

Denote by $\mathcal C^{ns}(\Sigma)$ the set of non-separating  curves.

\begin{thm}\label{thm: curve graph boundary}
There exists  a unique $G$-minimal closed subset denoted as $\partial G$ in the horofunction compactification $\overline{\mathcal C(\Sigma)}_h$    so that $\mathcal C^{ns}(\Sigma)=\partial G\cap \mathcal C(\Sigma)$. Moreover, the local minimum map  in Lemma \ref{localminimalmap} restricts to be injective over the Myrberg limit set in $\partial \mathcal C(\Sigma)$.     
\end{thm}
\begin{proof}
By a similar argument as in Lemma \ref{prop: minimal loxodromic elements}, there are pseudo-Anosov elements $f$ with minimal fixed points in $\partial_h \mathcal C(\Sigma)$. Let $\partial G$ be the closure of $\{gf^\pm: g\in G\}$ in the horofunction compactification $\overline{\mathcal C(\Sigma)}_h$. It is the minimal closed subset by Lemma \ref{limitsetbigtwo}.

By \cite{BirmanMen}, every separating curve is a dead end relative to a curve at distance 2 in $\mathcal C(\Sigma)$.  Thus, every separating curve is isolated in $\overline{\mathcal C(\Sigma)}_h$. The $\mathcal C^{ns}(\Sigma)=\partial G\cap \mathcal C(\Sigma)$ follows by Lemma \ref{lem: non-separating curve}.   The ``moreover" statement follows by Lemma \ref{lem: exist K local large sequence}.
\end{proof}

\subsection{Coned-off Cayley graphs of relatively hyperbolic groups}\label{subsec: coned-off cayley graphs}
Let $(G,\mathcal P)$ be a  finitely generated  group with  a finite collection of infinite  subgroups  $\mathcal P$ called peripheral subgroups. Fixing a finite generating set $S$, one can construct the   \textit{coned-off Cayley graph} $\hat \Gamma(G,S)$ as in \cite{Farb}  from the Cayley graph $\Gamma(G,S)$ as follows. For each left coset $gP$ with $g\in G$ and $P\in \mathcal P$, we  add an additional \textit{cone point} $c({gP})$  and then join $c({gP})$ to each element in $gP$ by an edge. Denote $\mathcal {C}(\mathcal P)=\{c({gP}): g\in G,P\in \mathcal P\}$ the set of the cone points. Then the vertex set of $\hat \Gamma(G,S)$ is the union of all group elements in $G$  and $\mathcal {C}(\mathcal P)$.  
A path $p$ has no \textit{backtracking} if it passes through each cone points at most once. We write $X=\hat \Gamma(G,S)$ for simplicity. We denote by $d$  the graph metric on $X$ and $d_S$ for the word metric on $\Gamma(G,S)$.

Similar to the subsurface projection, we may define a projection map $\pi_{gP}$ for each $gP$. Let  $x\in X$ be  a vertex which is either  a group element or a cone point.  Define $\pi_{gP}(x)$  to be the set of  vertices  on  some  geodesic from $x$ to $c(gP)$ which are adjacent to $c(gP)$. Write $d_{gP}(x,y):=\diam {\pi_{gP}(x)\cup\pi_{gP}(y)}$ where the diameter is measured using word metric $d_S$.

\begin{defn}\cite{Farb}
We say that $\hat \Gamma(G,S)$ satisfies  \textit{Bounded Coset Penetration} (BCP) property in sense of Farb if for any $c>1$, there exists a constant $K=K(c)$ with the following property. Let $p, q$ be two $c$-quasi-geodesics in $\hat \Gamma(G,S)$ without backtracking with same endpoints $p_-=q_-, p_+=q_+$. 
\begin{enumerate}
    \item   If $p$ and $q$ pass through a cone point $c(gP)$, then $d_{gP}(p_-,q_-)\le K$ and  $d_{gP}(p_+,q_+)\le K$. 
    \item   If $p$ passes a cone point $c(gP)$, but $q$ does not, then $d_{gP}(p_-,p_+)\le K$. 
\end{enumerate}   
The pair $(G,\mathcal P)$ is called \textit{relatively hyperbolic} if $\hat \Gamma(G,S)$ is a hyperbolic space and satisfies the BCP property. 
\end{defn}
In what follows, assume that  $(G,\mathcal P)$ is {relatively hyperbolic}.

Under the assumption that $\hat \Gamma(G,S)$ is hyperbolic,  the BCP (2) implies BCP (1). In line with Lemma \ref{lem: Masur Minsky BGI}, we shall use the following equivalent form of BCP property.
\begin{lem}\label{lem: BCP property}
There exists a constant $K>0$ with the following property.
Let $[x,c(gP)], [y,c(gP)]$ be two geodesics in $\hat \Gamma(G,S)$ from $x$ and $y$ to a cone point  $c(gP)$. If  $d_{gP}(x,y)>K$, then  any geodesic $[x,y]$ has to pass through $c({gP})$.    
\end{lem}
\begin{proof}
Assume that some geodesic $[x,y]$ does not contain $c({gP})$.
Using thin-triangle property, we may find vertices $x'\in p$ and $y'\in q$ so that $d(x',y')\le D$ and $d(x',x),d(y',y)>D$ for a constant $D$ depending on $\delta$. If the length of $p$ or $q$ is less than $D$, choose $x'=x$ or $y'=y$ accordingly. The concatenated path $[x',y'][y',c(gP)]$ is a  $D$-quasi-geodesic. Applying BCP(1) to the geodesic $[x',c(gP)]$ and $[x',y'][y',c(gP)]$, we derive $d_{gP}(x,y)\le K:=K(D)$. 
\end{proof}

Consequently, $\pi_{gP}(x)$ has $d_S$-diameter at most $K$, and the map $\pi_{gP}$ is $K$-Lipschitz: $d_{gP}(x,y)\le K d_X(x,y)$.

\begin{lem}\label{lem: no finite horofunction}
Let $b_\xi$ be a horofunction in $\bU_h\setminus \partial_h^\infty X$ so that $\min b_\xi>-\infty$. Then there exists a unique cone point $x\in \mathcal {Cone}(\mathcal P)$ so that the minimum of $b_\xi$ is realized  at $x$. In particular,   $\bU_h= X\cup \partial_h^\infty X$.     
\end{lem}
\begin{proof}
We only need show that the minimum of $b_\xi$ is realized  at a unique point $x$. This implies $b_\xi=b_x$ and thus the ``in particular" statement follows.

Now, let $x,y\in X$ so that $b_\xi(x)=\min b_\xi=b_\xi(y)$. For concreteness, assume that $d(o,x)\le d(o,y)$. Let $z_n\in X$ so that $b_{z_n}\to b_\xi$.  We first claim that $x$ must be  a cone point. Indeed, if $x$ is an element in $G$, then it has finite valence. Then $[x,z_n]$ will contain a common vertex $w$ adjacent to $x$ for all  $n\gg 1$, so $d(z_n,w)-d(z_n,o)=d(z_n,x)-d(z_n,o)-1$. This shows that $b_\xi(w)=b_\xi(x)-1$, contradicting the minimum assumption.
Thus,  $x$ is a cone point associated to a coset $gP$.  

Next observe that  $\{[z_n,x]\cap gP:n\ge 1\}$ contains infinitely many vertices. Indeed, if not, by taking a further subsequence, we may assume that for every $n\ge 1$, $[z_n,x]$ contains a common vertex $v\in gP$; that is, $d(z_n,v)+1=d(z_n,x)$. By computation, $b_{z_n}(v)=d(z_n,v)-d(z_n,o)=b_{z_n}(x)-1$, contradicting to $b_\xi(x)=\min b_\xi$. By Lemma \ref{lem: BCP property}, given any $y$, we have $[z_n,y]$ has to pass through $x$. This shows that $b_\xi(y)=b_\xi(x)+1$. Again a contradiction with the minimum. 
\end{proof}

The following is analogous to Lemma \ref{lem: exist K local large sequence}.
\begin{lem}\label{lem: cone points are not isolated}
The cone points in  $\mathcal {C}(\mathcal P)$ are accumulation points of $\partial_h^\infty X$.    
\end{lem}
\begin{proof}
Let $v=c(gP)$ be a cone point associated to the coset $gP$.  We are going to find a sequence of boundary points $\xi_n$ so that $\xi_n\to v$.  First, choose a geodesic ray $\gamma$ starting at $v$, which ends at a point $\xi$ in  the horofunction boundary  $\partial_h X$.  Then, let us pick a unbounded sequence of elements $p_n\in P$, so $\xi_n:=gp_ng^{-1} \xi$ is the end point of $gp_ng^{-1}\gamma$. We claim that $\xi_n\to v$. Indeed, for any vertex $w$ in $X$, let $u\in gP$ be the vertex on $[w,v]$ (which is adjacent to the cone point $v$). Since $gPg^{-1}$ acts properly on $gP$ in word metric $d_S$, we see that $d_{gP}(o,gp_n)\to \infty$ and $d_{gP}(u,gp_n)\to \infty$.  By Lemma \ref{lem: BCP property},  for all large enough $n\gg 1$, $v$ is a guard from $w$ to $\xi_n$.  The conclusion follows by Lemma \ref{lem: accumulation point criterion}.     
\end{proof}

The \textit{Bowditch boundary} denoted as $\partial_B G$ is  the union of cone points  and the Gromov boundary of $X$, equipped with the topology defined in \cite{Bow1}. That is, $\partial_B G=\mathcal {C}(\mathcal P)\cup \partial X$.

\begin{thm}\label{thm: coned graph boundary}
There is a unique $G$-invariant subset denoted as $\partial G$ in $\bU_h$ so that $\partial G\setminus \mathcal C(\mathcal P)$ is contained in $\partial_h^\infty X$. Moreover, the identity map on $\mathcal C(\mathcal P)$ extends to be a continuous map from $\partial G$  to  $\partial_B G$, which coincides with the local minimum map from $\partial G\setminus \mathcal C(\mathcal P)$ to the Gromov boundary $\partial X$.  
\end{thm}
\begin{proof}
By a similar argument as in Lemma \ref{prop: minimal loxodromic elements}, there are loxodromic elements $f$ with minimal fixed points in $\partial_h X$. Let $\partial G$ be the closure of  $\{gf^\pm: g\in G\}$ in the horofunction compactification $\bU_h$.   It is the minimal closed subset by Lemma \ref{limitsetbigtwo}. By Lemma \ref{lem: no finite horofunction}, we have $\partial G\setminus \mathcal C(\mathcal P)\subset \partial_h^\infty X$. The map in the ``moreover" statement is naturally defined and its continuity  is left as an exercise.
\end{proof}


\section{More $C^\ast$-algebra applications}
In this section, we give a concise and elementary proof of   the existence of free semigroup property, $P_{nai}$ property, and paradoxical towers in acylindrically hyperbolic groups. The proof uses only  the materials in \textsection \ref{sec: limit set and Myrberg} and may be read independently.   

Assume that a group $G$ acts non-elementarily on a hyperbolic space $X$. Let $$E(G)=\{g\in G: g\xi=\xi, \forall \xi\in \Lambda G\}$$ denote the set of elements which fix $\Lambda G$ pointwise. By definition, $E(G)$ is the kernel of the action of $G$ on $\Lambda G$, so $G/E(G)$ acts faithfully on $\Lambda G$. If the action is acylindrical, then $E(G)$ is the finite kernel studied in \cite{DGO}.
 
\subsection{Free semigroup property and  $P_{nai}$ property}

We start with some definitions.

The following definition shall play a crucial role in Claim \ref{claim: find f}. 
\begin{defn}\label{defn: tamed action}
A non-elementary action $G\act X$ is called \textit{tamed} if it contains a \textit{tamed}  pair of weakly independent loxodromic isometries $b,c\in G$ with the following property. 
Let $N(b,c):=\langle\langle b,c\rangle\rangle$ denote the normal closure  in $G$. Then either $g^\pm\cap f^\pm=\emptyset$ or $g^\pm= f^\pm$ for any two loxodromic isometries $g,f\in N(a,b)$.  

In particular, $\{gb^\pm: g\in G\}$ and $\{gc^\pm: g\in G\}$ are disjoint.
\end{defn}
\begin{rem}
The non-elementary acylindrically hyperbolic action $G\act X$ is tamed. Indeed,  by   \cite[Theorem 6.8]{DGO} (recalled in Lemma \ref{lem: acylin independent}), any pair $(b,c)$ of independent loxodromic isometries, $\{gb^\pm: g\in G\}$ and $\{gc^\pm: g\in G\}$ are disjoint. Thus, every pair of non-commensurable loxodromic isometries   are tamed in $G$.
\end{rem}

The following is the main result of this subsection.

\begin{prop} \label{prop: free semgigroups}
Assume that the non-elementary action $G\act X$ is tamed. Let $A$ be a finite set of non-trivial elements in $G$ so that $A\cap E(G)=\emptyset
$.  Then 
\begin{enumerate}[label=(\roman*)]
    \item
    There exists $f\in G$ so that $A f$ generates a free semi-group.
    
    \item
    There exists $f\in G$ so that $\langle a, f\rangle$ is a free group of rank 2 for each $a\in A$.
\end{enumerate}
\end{prop}
\begin{rem}
The element $f$ in the two items may not be the same. The asserted conclusions in (i) and (ii) are referred in literature as free semigroup property and $P_{nai}$ property.    In acylindrically hyperbolic groups,  \textit{(i)} is proved in \cite[Proposition 4.1(a)]{GO20}  and \textit{(ii)} is proved in \cite[Theorem 0.1]{GO20}.   
\end{rem}
  
\begin{proof}
We first start with two ingredients which use the tamed action in an essential way. 
\begin{claim}\label{claim: find f}
There is a loxodromic isometry $f\in G$ so that $af^\pm\cap f^\pm=\emptyset$ for each $a\in A$.     
\end{claim}
\begin{proof}[Proof of the Claim]
Note that $\Lambda G=\overline{\{gh^\pm: g\in G\}}$ for any loxodromic isometry $h\in G$. Let $(b,c)$ be a tamed pair of weakly independent loxodromic isometries in Definition \ref{defn: tamed action}.   Since $A$ is disjoint with $E(G)$, given $a\in A$,  there are conjugates $h_1, h_2\in G$ of $b,c$   so that   $ah_1^\pm\cap h_1^\pm=\emptyset$ and $ah_2^\pm\cap h_2^\pm=\emptyset$.  By Lemma \ref{lem: double dense elements}, $g_n:=h_2^nh_1^n$ for $n\gg 0$ are  loxodromic isometries  whose fixed points tend to $h_2^+$ and $h_1^-$. Noting that  $g_n$ are contained in the normal closure $N(b,c)$,  we see that $ag_n^\pm\cap g_n^\pm=\emptyset$ for all large enough $n\gg 0$.   Since $A$ is a finite set, a repeated argument as above produces a desired $f\in N(b,c)$: for any $a\in A$, $af^\pm\cap f^\pm=\emptyset$.  
\end{proof}

\begin{claim}\label{claim: bdd proj}
Let $f$ be given by Claim \ref{claim: find f}. Then there exist a bounded projection constant  $\tau>0$  so that
\begin{enumerate}
    \item 
    $\mathrm{diam}(\pi_{\ax(f)}a\ax(f))\le \tau$ for any $a\in A\cup A^{-1}$.
    \item 
    $\mathrm{diam}(\pi_{\ax(f)}[o,a^no])\le \tau$ for any  $a\in A\cup A^{-1}$ and any $0\ne n\in \mathbb Z$.
\end{enumerate}    
\end{claim}
\begin{proof}[Proof of the Claim]
The (1) follows from Claim \ref{claim: find f} that $\ax(f)\ne a\ax(f)$ with $a\in A$. We now prove (2). We examine the following three cases according to the type of $a\in A$.
If $a$ is an elliptic isometry, then $\langle a\rangle o$ has bounded diameter and thus the bounded projection follows.

Assume now that $a$ is a loxodromic or parabolic  isometry. Since  $\ax(f)\ne a\ax(f)$ and $f$ is contained in $N(b,c)$, the fixed points of $a$ are disjoint with $f^\pm$. Since $\langle a\rangle o$ is a discrete orbit in $X$, the bounded projection of $\langle a\rangle o$ to $\ax(f)$ follows: the value of $\tau$ depends on $a$, but the set $a\in A$ is finite, so the (2) is proved.
\end{proof}

Since $f$ is loxodromic, we now take a large power so that $d(o,f^no)\gg \tau$ for some $n\gg 1$. In what follows, the power is raised as necessary and we write $f=f^n$ for simplicity. This implies that $a_1fa_2f\cdots a_m f$ for any $a_1, a_2, \cdots, a_m\in A\cup A^{-1}$ labels a $\lambda$-quasi-geodesic for some $\lambda$ depending on $\tau$. 

We are now ready to give the proof of the two items. 

\textit{(i)}.
Let $\tilde A$ denote the set of elements $a^{-1}b$ for $a\ne  b\in A$. The desired element $f$ is given by Claim \ref{claim: find f} applied to $\tilde A$.
It remains to show $\tilde A f$ generates a free semigroup. To seek a contradiction, for two distinct sequences $a_1, \cdots, a_n\in \tilde A$ and $a_1', \cdots, a_n'\in \tilde A$, assume 
$$a_1fa_2f\cdots a_m f= a_1'fa_2'f\cdots a_n'f$$ so that the two words on the left and right  sides  label two $c$-quasi-geodesics with the same endpoints. Without loss of generality, assume $a_1\ne a_1'$ and $d(o,a_1o)\ge d(o,a_1'o)$.  

Since $a_1'^{-1}a_1\in \tilde A$, the item {(1)} of Claim \ref{claim: bdd proj} implies that $[o,a_1'^{-1}a_1o]$  has $\tau$-bounded projection to $\ax(f)$. By Morse Lemma, $d(a_1'o,[o,a_1o])\le R$ for some $R=R(c)$.  These two facts imply  $|d(o,a_1o)-d(o,a_1'o)|\le D$ for a constant $D$ depending on $R,\tau$. However,  $a_1'^{-1}a_1\ax(f)$ also has $\tau$-bounded projection to $\ax(f)$ by Claim \ref{claim: bdd proj}(2). This   gives a contradiction   when $d(o,fo)$ is large. This shows that $Af$ generates a free semi-group.

\textit{(ii)}. The desired element $f$ is given by Claim \ref{claim: find f} for $A$.  
By Claim \ref{claim: bdd proj}(2), $\mathrm{diam}(\pi_{\ax(f)}[o,a^no])\le \tau$ for  $n\in\mathbb Z\setminus 0$, so the above argument proves that $\langle a, f\rangle$ is a free group of rank 2 for any $a\in A$.
\end{proof}

Let us  conclude this part by discussing the connection to the work of de la Harpe and his collaborators. Let us consider an action of $G$ by homeomorphism on a Hausdorff topological space $Z$.
\begin{enumerate}
    \item It is called \textit{strongly faithful} if for any finite set of non-trivial elements $F$ in $G$, there exists a point $\xi\in Z$ so that $f\xi\ne\xi$ for $f\in F$. 
    \item 
    A \textit{hyperbolic element}  admits north-south dynamics relative to its two fixed points. Two hyperbolic elements are called \textit{transverse} if their fixed points are disjoint. 
    \item 
    The action is called \textit{strongly hyperbolic} if $G$ contains at least a pair of transverse hyperbolic elements. 
\end{enumerate} 
It is proved that topological free action on $Z$ is strongly faithful (\cite[Corollary 10]{HJ11}). If a group admits a minimal, strongly faithful and strongly hyperbolic action, then the group is Powers' group (so it is $C^\ast$-simple)  and has free semi-group property \cite[Lemma 4.3]{DH99}. 

It is clear that a non-elementary action on a hyperbolic space $X$ induces a strongly hyperbolic action on $\partial X$. In the Claim \ref{claim: find f}, we proved that if the action is tamed and $E(G)=\{1\}$,   it  is strongly faithful. In view of this,  Proposition \ref{prop: free semgigroups}(i)  would then follow from \cite[Lemma 4.3]{DH99}.

\subsection{Paradoxical towers}
The notion of paradoxical towers is introduced in \cite{GGKN}.
\begin{defn}Let $n \in N$ be an integer. We say that a countable group $G$ admits \textit{$n$-paradoxical towers} if for every non-empty finite subset $D \subset G$, there are $A_1, \cdots, A_n \subset G$ and $g_1,\cdots, g_n\subset G$ such
that:
\begin{enumerate}
    \item 
    the sets $aA_i$, for each $a \in D$ and $1\le i\le n$ are pairwise disjoint,
    \item 
    $G =\cup_{i=1}^n g_iA_i$.
\end{enumerate}

We say that $G$ admits \textit{paradoxical towers} if there is $n \in\mathbb  N$ such that $G$ admits $n$-paradoxical towers.
\end{defn}

The following result in acylindrically hyperbolic groups is proved  \cite[Proposition 4.7]{GGKN}, which uses the fact that topologically free extreme boundary actions are {$2$-paradoxical towers} by \cite[Proposition 4.6]{GGKN}.
\begin{prop}\label{prop: paradoxical towers}
Assume that the non-elementary action $G\act X$ is tamed and $E(G)$ is trivial. Then $G$ admits {$3$-paradoxical towers}.    
\end{prop}
\begin{proof}
Fix a set  $F=\{f_1,f_2,f_3\}$ of three weakly independent loxodromic elements.
Let $D$ be a finite set of elements in $G$. Apply Claim \ref{claim: find f} to the union of $D$ and $D^{-1}D$. Then there exists $f_0\in G$ so that $a f_0^\pm\cap f_0^\pm=\emptyset$ and $a^{-1}a'f_0^\pm\cap f_0^\pm=\emptyset$ for any $a\ne a'\in D$. 

Let $c>1$   be given by Lemma \ref{extend3} for the set $F$. Given $a\in D$ and $b\in G$, there exists $f_{b}\in F$ so that the path labeled by the word $a\cdot f_0\cdot f_{b} \cdot b$ :  $$[o,ao]a[o,f_0o] af_0[o,f_{b}o]af_0f_{b}[o,bo]$$  is a $c$-quasi-geodesic.
\begin{claim}
If $a\ne a'$ and $b\ne b'$, then $a\cdot f_0\cdot f_{b}\cdot b\ne a' \cdot f_0 \cdot f_{b'}\cdot b'$.    
\end{claim}
\begin{proof}[Proof of the claim]
Assume first that $a=a'$ and $f_{b}\ne f_{b'}$; otherwise there is nothing to do. Thus, their axis of $f_b$ and $f_{b'}$ have bounded projection, so $b^{-1}f_{b}^{-1}f_{b'}b'$ labels a $c$-quasi-geodesic $\gamma$. Taking a larger power of $f\in F$ so that $d(o,f_b^no)\gg 0$, this gives a contradiction with the fact that $\gamma$ has the same endpoints.  

Assume now that $a\ne a'$. Since $a^{-1}a'f_0^\pm\cap f_0^\pm=\emptyset$,  the word  $b^{-1}f_{b}^{-1}(a^{-1} a')f_0f_{b'}b'$ labels a $c$-quasi-geodesic $\gamma$, which is a loop. A similar argument as above concludes the proof in this case.
\end{proof}

Let $A_1$ be the set of elements $f_0 f_{b}b$ with $b\in G$ so that $f_{b}=f_1$. $A_2$ and $A_3$ are similarly defined. By the claim, $aA_1,aA_2,aA_3$ for every $a\in D$ are all disjoint. Let $g_i=f_i^{-1}f_0^{-1}$ with $1\le i\le 3$. Then $G=g_1 A_1\cup g_2A_2\cup g_3A_3$.
\end{proof}



\bibliographystyle{amsplain}
 \bibliography{bibliography}

\end{document}